\documentclass[a4paper,leqno]{article}

\usepackage[utf8]{inputenc}

\usepackage{microtype}
\usepackage{amscd}
\usepackage{amsmath,amssymb,amsfonts,accents}
\usepackage{xstring}
\usepackage{xcolor}
\usepackage[mathscr,mathcal]{euscript}
\usepackage{xypic}
\xyoption{rotate}
\usepackage[all,v2,cmtip,2cell]{xy}
\UseAllTwocells
\entrymodifiers={+!!<0pt,\fontdimen22\textfont2>}
\usepackage{url}
\usepackage{latexsym}
\usepackage{amsthm}
\usepackage{multicol}
\usepackage{enumitem}
\usepackage[colorlinks=true, linkcolor={blue!50!black}, pdfhighlight=/O, ocgcolorlinks=true]{hyperref}
\usepackage{graphicx}
\usepackage{color}
\usepackage{mathtools}
\usepackage[colorinlistoftodos]{todonotes}
\usepackage[affil-it]{authblk}

\usepackage{adjustbox}

\usepackage{cleveref}

\usepackage{tikz}

\usepackage{tikz-cd}
\usetikzlibrary{matrix,arrows, patterns, positioning, calc, intersections}

\usepackage{tikz-3dplot}

\allowdisplaybreaks

\usepackage[a4paper, total={5.5in,8.5in}]{geometry}

\numberwithin{equation}{section}
\theoremstyle{plain}
\newtheorem{theorem}[equation]{Theorem}
\newtheorem{lemma}[equation]{Lemma}
\newtheorem{proposition}[equation]{Proposition}
\newtheorem{corollary}[equation]{Corollary}

\newtheorem{ThmA}{Theorem}

\theoremstyle{definition}

\newtheorem{definition}[equation]{Definition}
\newtheorem{construction}[equation]{Construction}
\newtheorem{example}[equation]{Example}
\newtheorem{examples}[equation]{Examples}
\newtheorem{remark}[equation]{Remark}

\newtheorem{notation}[equation]{Notation}

\setlist[enumerate]{label=(\arabic*), leftmargin=*}
\setenumerate{label=(\arabic*), leftmargin=*}
\setlist[itemize]{label=$\vcenter{\hbox{\footnotesize$\bullet$}}$, leftmargin=*}
\setitemize{label=$\vcenter{\hbox{\footnotesize$\bullet$}}$, leftmargin=*}

\newcommand{\mb}[1]{\mathbf{#1}}
\newcommand{\mm}[1]{\mathrm{#1}}

\newcommand{\cat}[1]{
\StrLen{#1}[\mystrlen]
\ifnum\mystrlen=1 \mathscr{#1}
\else \mathrm{#1}
\fi}
\newcommand{\scat}[1]{\mb{#1}}

\newcommand{\colim}{\operatornamewithlimits{\mathrm{colim}}}

\newcommand{\Map}[0]{\mm{Map}}

\newcommand{\Fun}[0]{\cat{Fun}}

\newcommand{\op}[0]{\mm{op}}

\newcommand{\bPerf}{\mathbf{Perf}}

\newcommand{\cB}{\mathcal{B}}
\newcommand{\cA}{\mathcal{A}}
\newcommand{\cC}{\mathcal{C}}
\newcommand{\cM}{\mathcal{M}}
\newcommand{\cG}{\mathcal{G}}

\newcommand{\cS}{\mathcal{S}}
\newcommand{\cLoc}{\mathcal{L}}

\newcommand{\Spec}{\mathrm{Spec}}

\makeatletter
\newcommand{\mytag}[2]{%
  \text{#1}%
  \@bsphack
  \begingroup
    \@onelevel@sanitize\@currentlabelname
    \edef\@currentlabelname{%
      \expandafter\strip@period\@currentlabelname\relax.\relax\@@@%
    }%
    \protected@write\@auxout{}{%
      \string\newlabel{#2}{%
        {#1}%
        {\thepage}%
        {\@currentlabelname}%
        {\@currentHref}{}%
      }%
    }%
  \endgroup
  \@esphack
}
\makeatother

\title{Topological field theories associated with Calabi--Yau categories}

\author[]{
	Tristan Bozec\thanks{LAREMA, Univ. Angers, CNRS, Angers, France \\
											\href{mailto:tristan.bozec@univ-angers.fr}{tristan.bozec@univ-angers.fr}}, 
	Damien Calaque\thanks{IMAG, Univ. Montpellier, CNRS, Montpellier, France \\
											\href{mailto:damien.calaque@umontpellier.fr}{damien.calaque@umontpellier.fr}}, 
	Sarah Scherotzke\thanks{Mathematical Institute, University of Luxembourg, Luxembourg \\
											\href{mailto:sarah.scherotzke@uni.lu}{sarah.scherotzke@uni.lu}}}

\date{}

\begin{document}

\maketitle

\begin{abstract}
We construct symmetric monoidal higher categories of iterated Calabi--Yau cospans, that are noncommutative analogs of iterated lagrangian correspondences. 
We actually give a general (and functorial) procedure that applies to iterated nondegenerate cospans on certain comma categories. 
This allows us to factor the AKSZ fully extended TFT associated with the moduli of objects of a Calabi--Yau category (taking values in iterated lagrangian correspondences) through a fully extended TFT taking values in iterated Calabi--Yau cospans. 
\end{abstract}

\setcounter{tocdepth}{2}
\tableofcontents

\makeatletter
\section*{Introduction}
\addcontentsline{toc}{section}{Introduction}

\subsection*{Motivation}

\paragraph{AKSZ construction.}
The AKSZ construction is a transgression to mapping spaces that was studied in \cite{AKSZ}. In this context one works with $QP$-manifolds and the BV formalism, 
and one transgresses a solution $S$ of the classical master equation. This construction is relevant to quantum field theory: the critical locus of $S$ happens to be a 
derived enhancement of the moduli space of classical solutions of the equations of motion, over which functional integrals localize. 

\medskip

\paragraph{Shifted symplectic structures.}
In various field theories, the moduli space of solutions of the equations of motion does not appear as a global critical locus (unless one goes through infinite dimensional 
constructions). Nevertheless, such moduli always admit local critical charts. The structure that survives is the one known as a $(-1)$-shifted structure, introduced in the 
pionneer work \cite{PTVV}, and it suffices for producing counting invariants informally defined \textit{via} functional integrals (see e.g.~\cite{B2J,B4J}). 

\medskip

The AKSZ construction (or, transgression to mapping spaces) has been worked out in \cite{PTVV} for $s$-shifted symplectic structures on derived Artin stacks, in the 
context of derived algebraic geometry. It was then extended in \cite{Cal} to boundary conditions, leading to the existence, for every $s$-shifted symplectic stack $X$, 
of an oriented topological field theory of dimension $n$ with values in the homotopy category of $(s-n+1)$-shifted symplectic stacks and lagrangian correspondences 
thereof. Let us finally mention that the theory of shifted symplectic structures as well as the AKSZ construction actually make sense more generally for derived prestacks 
having a perfect cotangent complex, as shown in \cite{CSa}. 

\medskip

\paragraph{The AKSZ construction as a fully extended TFT.}
It was conjectured in \cite{Cal} and proved in \cite{CHS} that these topological field theories are actually (fully) extended, in the sense of Baez--Dolan \cite{BD-TFT} 
and Lurie \cite{Lurie-TFT}. More precisely, the results of \cite{CHS} can be summarized as follows: 
\begin{itemize}
\item[--] There exists a symmetric monoidal $(\infty,n)$-category $\scat{Lag}_n^s$ whose objects are $s$-shifted symplectic derived stacks and 
whose $i$-morphisms for $i \leq n$ are $i$-fold Lagrangian correspondences. This is \cite[Theorem A]{CHS}.
\item[--] Similarly, there exists a symmetric monoidal $(\infty,n)$-category $\scat{Or}_n^d$ whose objects are $d$-oriented derived stacks and 
whose $i$-morphisms for $i \leq n$ are $i$-fold oriented cospans. 
\item[--] For every $s$-shifted symplectic stack $X$, the mapping stack functor $\textsc{Map}(-,X)$ can be upgraded to a symmetric monoidal $(\infty,n)$-functor 
\[
\mathrm{aksz}_X\,:\,\scat{Or}_n^d\,\longrightarrow\, \scat{Lag}_n^{s-d}
\]
using transgression. This is \cite[Theorem C]{CHS}.
\item[--] The Betti stack construction $(-)_B$ can be promoted to a symmetric monoidal $(\infty,n)$-functor  
\[
\mathrm{orcut}_{0,n}\,:\,\scat{Bord}_{0,n}^{\mathrm{or}}\,\longrightarrow\, \scat{Or}_n^0
\]
where $\scat{Bord}_{0,n}^{\mathrm{or}}$ is the symmetric monoidal $(\infty,n)$-category of iterated oriented cobordisms from \cite{Lurie-TFT} (see also \cite{CS}). 
This is \cite[Theorem E]{CHS}. 
\end{itemize}
All in all, this tells us that, for every $s$-shifted symplectic derived stack $X$, the functor $\textsc{Map}\big((-)_B,X\big)$ gives rise to an $n$-dimensional oriented fully 
extended topological field theory with values in $\scat{Lag}_n^s$, given as $\mathrm{aksz}_X\circ \mathrm{orcut}_{0,n}$. 

\medskip

\paragraph{A noncommutative AKSZ construction.}
In \cite{ToEMS}, a version of the transgression procedure/AKSZ construction for noncommutative spaces (i.e.~dg-categories) was suggested. 
It has been fully developped in \cite{BD2} and extended to the relative case in \cite{BD1,BD2} (see also \cite{ToCY}). 
This allows to get: 
\begin{itemize}
\item[--] A symmetric monoidal ($1$-)category $h\mathrm{CY}_1^d$ whose objects are $d$-Calabi--Yau dg-categories and whose morphisms 
are weak equivalence classes of Calabi--Yau cospans. This essentially follows from \cite[Theorem 6.2]{BD1}. 
\item[--] A symmetric monoidal functor 
\[
h\mathrm{CY}_1^d\,\longrightarrow\,h\mathrm{Lag}_1^{2-d}
\]
lifting the one that sends a dg-category $\cA$ to its moduli of objects $\mathbf{Perf}_{\cA}$ in the sense of \cite{ToVa}. This essentially follows from \cite[Theorem 5.5]{BD2} (even 
though compatibility with composition isn't proven in \emph{loc. cit.}). 
\item[--] A symmetric monoidal ($1$-)category $h\mathrm{Poinc}_1^d$ whose objects are $d$-Poincaré complexes and whose morphisms are 
weak equivalence classes of Poincaré cospans. The construction of this category can be guessed from \cite[\S\,5.1]{BD1}. 
\item[--] A symmetric monoidal functor 
\[
\mathcal{L}\,:\,h\mathrm{Poinc}_1^d\,\longrightarrow\,h\mathrm{CY}_1^d\,.
\]
The construction for objects and morphisms is given in \cite[\S\,5.1]{BD1}, but it is not proven to define a functor. 
\end{itemize}

\subsection*{Results}

\paragraph{Expectations.}
From what we have seen so far, a clear motivation for this paper is to ``fully extend'' the above noncommutative AKSZ construction\footnote{Observe that Calabi--Yau categories can either be seen as noncommutative oriented spaces or as 
noncommutative shifted symplectic stacks; hence both $\mathcal L$ and $\mathbf{Perf}$ would deserve to be considered as a noncommutative analog of the AKSZ construction. We (arbitrarily) choose the former. }, showing that: 
\begin{itemize}
\item[--] The functor $\mathbf{Perf}$ can be upgraded to a symmetric monoidal $(\infty,n)$-functor 
from a higher category $\scat{CY}_n^d$ of iterated Calabi--Yau cospans to $\scat{Lag}_n^{2-d}$.  
\item[--] The functor $\mathcal L$ can be upgraded to a symmetric monoidal $(\infty,n)$-functor 
\[
\mathrm{ncaksz}_k\,:\,\scat{Poinc}_n^d\,\longrightarrow\, \scat{CY}_n^d
\]
where $\scat{Poinc}_n^d$ is the sub-$(\infty,n)$-category of $\scat{Or}_n^d$ spanned by iterated oriented cospans 
only involving constant stacks of spaces of finite type (so that the functor $\mathrm{orcut}_{0,n}$ factors through $\scat{Poinc}_n^0$). 
\end{itemize}

Two observations are in order: 
\begin{enumerate}
\item In the case $d=0$, the existence of the fully extended TFT $\mathrm{ncaksz}_k$ (or, rather, $\mathrm{ncaksz}_k\circ \mathrm{orcut}_{0,n}$) was conjectured in \cite[Remark 3.10]{BCS2}. 
In Section 5 of \textit{loc. cit.}, some features of this conjectural fully extended TFT were used to guess a comparison between the original AKSZ construction and its noncommutative analog in dimension $2$. 
All this suggests that $\mathbf{Perf}\circ\mathrm{ncaksz}_k=\mathrm{aksz}_{\mathbf{Perf}_k}$. 
\item The functor $\mathrm{ncaksz}_k$ sends the point to $k$. It is expected that everything still works if we replace $k$ with an arbitrary $s$-Calabi--Yau dg-category $\cC$ (similarly, one can feed the original AKSZ construction with any 
shifted symplectic derived stack): there should be a symmetric monoidal $(\infty,n)$-functor 
\[
\mathrm{ncaksz}_{\cC}\,:\,\scat{Poinc}_n^d\,\longrightarrow\, \scat{CY}_n^{d+s}
\]
such that $\mathbf{Perf}\circ\mathrm{ncaksz}_{\cC}=\mathrm{aksz}_{\mathbf{Perf}_{\cC}}$. 
\end{enumerate}

The goal of this paper is to prove the above expectations are correct. 

\paragraph{Higher categories of iterated Calabi--Yau cospans.}
Our first task is to construct the symmetric monoidal higher category of Calabi--Yau cospan, that is the codomain of the noncommutative AKSZ functor: 
\begin{ThmA}\label{TheoremA}
There is a symmetric monoidal $(\infty,n)$-category $\scat{CY}_n^d$ whose objects are $d$-Calabi--Yau categories and whose $i$-morphisms
for $i \leq n$ are $i$-fold Calabi--Yau cospans. 
\end{ThmA}
The construction of this symmetric monoidal higher category is completely similar to the one of the symmetric monoidal higher category of iterated lagrangian correspondences from \cite{CHS}. 
As in \textit{loc. cit.}, the most technical part of the construction is when one deals with the non-degeneracy condition for iterated (co)spans. We provide a systematic general treatment, that has two advantages: 
\begin{enumerate}
\item It is not specific to the examples we have to deal with, and can potentially be used in other contexts. 
\item It is functorial, allowing to obtain the comparison statements we are looking for. 
\end{enumerate}

\paragraph{Noncommutative AKSZ is fully extended.}
Our second task is to prove that every Calabi--Yau category gives rise to an oriented fully extended topological field theory. 
This is obtained as a consequence of the following: 
\begin{ThmA}\label{TheoremB}
For every $s$-Calabi--Yau dg-category $\cC$, the noncommutative AKSZ construction gives a symmetric monoidal $(\infty,n)$-functor 
\[
\mathrm{ncaksz}_{\cC}\,:\,\scat{Poinc}_n^d\,\longrightarrow\, \scat{CY}_n^{d+s}\,.
\]
\end{ThmA}
The desired fully extended TFT will be obtained, in the case $d=0$, as the composed functor 
\[
\mathrm{ncaksz}_{\cC}\circ \mathrm{orcut}_{0,n}\,:\,\scat{Bord}_{0,n}^{\mathrm{or}}\,\longrightarrow\, \scat{CY}_n^{s}\,.
\]

\medskip

One shall observe that $\mathrm{ncaksz}_{\cC}$ is obtained by composing the symmetric monoidal $(\infty,n)$-functor 
\[
\mathrm{ncaksz}_{k}=\mathcal{L}\,:\,\scat{Poinc}_n^d\,\longrightarrow\, \scat{CY}_n^{d}\,.
\]
with a symmetric monoidal $(\infty,n)$-functor 
\[
\scat{CY}_n^{d}\,\longrightarrow\, \scat{CY}_n^{d+s}
\]
given by tensoring with $\cC$ (that we will construct). 

\paragraph{Comparing AKSZ with its noncommutative counterpart.}
Every fully extended TFT with values in iterated Calabi--Yau cospans leads to a fully extended TFT with values in iterated lagrangian correspondence, thanks to the following: 
\begin{ThmA}\label{TheoremC}
The moduli of objects functor $\mathbf{Perf}$ can be promoted to a symmetric monoidal functor (that we still denote $\mathbf{Perf}$)
\[
\scat{CY}_n^{d}\,\longrightarrow\,\scat{Lag}_n^{2-d}\,.
\]
\end{ThmA}

Finally, we will show that, unsurprisingly, the above sends the noncommutative AKSZ TFT associated with an $s$-Calabi--Yau category $\cC$ to the AKSZ TFT associated with the $(2-s)$-shifted symplectic stack $\mathbf{Perf}_{\cC}$: 
\begin{ThmA}\label{TheoremD}
There is a natural equivalence 
\[
\mathbf{Perf}\circ\mathrm{ncaksz}_{\cC}\cong \mathrm{aksz}_{\mathbf{Perf}_{\cC}}
\]
of symmetric monoidal $(\infty,n)$-functors. 
\end{ThmA}

The following commuting diagram of symmetric monoidal $(\infty,n)$-categories summarizes our constructions and results: 
\[
\begin{tikzcd}[column sep=5em,row sep=3em]
\scat{Poinc}_n^d \arrow[hookrightarrow]{dr} \arrow{r}{\mathcal{L}=\mathrm{ncaksz}_k} \arrow[bend left=30]{rr}{\mathrm{ncaksz}_{\cC}} & \scat{CY}_n^{d}   \arrow{r}{-\otimes\cC} & \scat{CY}_n^{d+s} \arrow{d}{\mathbf{Perf}} \\
\scat{Bord}_{d,d+n}^{\mathrm{or}}\arrow[dashed]{u}{\mathrm{orcut}_{d,n}} \arrow[swap]{r}{\mathrm{orcut}_{d,n}} & \scat{Or}_n^d \arrow[swap]{r}{\mathrm{aksz}_{\mathbf{Perf}_{\cC}}} & \scat{Lag}_n^{2-(d+s)}
\end{tikzcd}
\]
We have incorporated the slight generalization $\mathrm{orcut}_{d,n}$ (also introduced in \cite{CHS}) of $\mathrm{orcut}_{0,n}$, that factors through $\scat{Poinc}_n^d\subset \scat{Or}_n^d$. Fully extended TFTs correspond to the case $d=0$.  

\medskip

Throughout this paper $k$ is a field of characteristic zero.

\subsection*{Outline of the paper}

In a first section we recall the necessary material about iterated (co)spans in the framework of $\infty$-categories, as settled by Haugseng in~\cite{Haug}, including monoidal structures.
We explain and detail the three fundamental examples that arise in our main results. Precisely we will deal with the topological notion of prePoincaré spaces, the geometric notion of (almost) pre-symplectic stacks, and the algebraic notion of (almost) 
pre-Calabi--Yau structures. All correspond to comma categories associated with the following functors: singular homology, $2$-forms, closed $2$-forms, and Hochschild and negative cyclic homology.

The second section is devoted to the implementation of non-degeneracy conditions in the above recalled context. This requires the technology of twisted arrow categories, that we treat in a `relative' fashion: the aim is to use 
natural transformations between our main examples to cook up functors between the associated absolute constructions of non-degenerate iterated spans. 
We keep track of the symmetric monoidal structures, and use some more or less standard folklore about Grothendieck constructions that we recall in \cref{appendixA}.
More precisely, we construct symmetric monoidal higher categories of iterated Poincaré cospans, iterated Calabi--Yau cospans and iterated lagrangian spans, and we show that 
$\mathfrak{L}$ and $\mathbf{Perf}$ can be promoted to symmetric monoidal higher functors (from Poincaré cospans to Calabi--Yau cospans, and from Calabi--Yau cospans to lagrangian spans, respectively).

In a third part we pack everything up in order to prove our announced theorems. It requires technical lemmas which prove compatibility between various intricate constructions involving tensor products, mapping stacks and symplectic geometry. 
Precisely, we check that the mapping stack construction naturally lifts to our non-degenerate iterated spans, in a way compatible with what was previously done.

There are also two more appendices where we use the material from the second section: \\[-0.6cm]
\begin{itemize}
\item[--] In \cref{appendixC} we briefly explain how the same methods allow to construct higher symmetric monoidal categories of iterated \emph{right} Calabi--Yau \emph{spans}. \\[-0.6cm]
\item[--] In \cref{appendix} we prove a lemma announced in~\cite{BCS}, but in a very general way that we believe may be of interest for the community as it already appears in shifted symplectic geometry.
\end{itemize}

\subsection*{Related works}

\paragraph{Cubical Calabi--Yau structures.}
In \cite[\S6.4]{CDW} the authors introduce a notion of Calabi--Yau structure on $N$-cubical diagrams of dg-categories. 
They are examples of $N$-morphisms in the iterated Calabi--Yau cospan higher category for which several $i$-arrows with $i<N$ are trivial. 
For instance, a Calabi--Yau square ($N=2$) 
\[
\xymatrix{
\cat{C}_{11}\ar[r]\ar[d] & \cat{C}_{01}\ar[d] \\
\cat{C}_{10}\ar[r] & \cat{C}_{00}
}
\]
in the sense of \cite[Definition 6.4.3]{CDW} is exactly a $2$-morphism of the form  
\[
\begin{tikzpicture}
\node (a) at (0,0) {$\emptyset$};
\node (b) at (1.5,1) {$C_{00}$};
\node (c) at (3,0) {$\emptyset$};
\draw[->] (a) to [bend left] node[scale=.7] (f) [above] {$C_{01}$~~~~} (b);
\draw[->] (b) to [bend left] node[scale=.7] (g) [above] {~~~~$C_{10}$} (c);
\draw[->] (a) to [bend right] node[scale=.7] (h) [below] {$\emptyset$} (c);
\draw[-{Implies},double distance=1.5pt,shorten >=2pt,shorten <=2pt] (b) to node[scale=.7] [right] {~$C_{11}$} (h);
\end{tikzpicture}
\] 
in the iterated Calabi--Yau cospan higher category. 

\medskip

\paragraph{Oriented constructible cosheaves of categories.}
The $n$-morphisms in the iterated Calabi--Yau cospan higher category are exactly the ``oriented ccc spaces'' of 
\cite[Definition 4.24]{ST} whose underlying stratified space is the Poincaré dual to the $n$-globe (for instance, if $n=2$, then the Poincaré dual to the $2$-globe 
\raisebox{-7pt}{$
\begin{tikzpicture}
every node/.style={fill=black,circle},
\node at (0,0) [circle,fill,inner sep=1pt]{};
\node at (1,0) [circle,fill,inner sep=1pt]{};
\draw[-] (0,0) to [bend left] node[scale=.9] [above] {} (1,0);
\draw[-] (0,0) to [bend right] node[scale=.9] [below] {} (1,0);
\draw[double equal sign distance] (0.5,0.12) to (0.5,-0.12);
\end{tikzpicture}
$}
is the stratified space 
$\{(0,0)\}\subset\{0\}\times\mathbb{R}\subset\mathbb{R}^2$). 
Similarly, the $n$-morphisms in the iterated lagrangian correspondence higher category from \cite{CHS} are exactly the constructible families of shifted symplectic 
derived stacks from \cite{APT}, over the Poincaré dual to the $n$-globe. 
This actually suggests that there should be an alternative approach to \cite{CHS} and the present paper, based on constructible families 
(of Poincaré spaces, of derived stacks, of dg-categories, etc...). 

\medskip

\paragraph{Cohomological Hall algebras for $3$-manifolds.}

In \cite[Section 9]{KPS} the authors explain how to define cohomological Hall algebras associated with $3$-manifolds, using general results 
about $3$-Calabi--Yau categories. To achieve this they also compare (as we do) the moduli of objects of the dg-category of local systems on a Poincaré duality 
space $M$ with the mapping stack from $M_B$ to the moduli of perfect complexes, and their symplectic structures. 

The relevant parts of \cite{KPS} are proposition 8.24, theorem 8.25 as well as the functoriality discussed just after, and theorem 9.3. 
Proposition 8.24 and theorem 8.25 are very general and apply to our setting with (using their notation) $X=M_B$, 
$\mathcal D=\mathcal L(M)=\scat{QCoh}(X)$ and $E=E_M$ as in the proof of their theorem 9.3. 

Their proposition 8.24 and its functoriality give back our \cref{lemma:objects} (hence, their result is more general). 
Their theorem 9.3 (that uses their theorem 8.25) corresponds to our \cref{lemma:symplectic} when $\mathcal C=k$ (hence, our result is more general -- even though 
their proof can probably be adapted to the general case).

\medskip

Still in the context of Hall algebras we want to point out that a \emph{two}-category of (almost) \emph{right} Calabi--Yau\footnote{Almost Calabi--Yau structures are also known as \emph{weak} Calabi--Yau structures. 
Our terminology is inspired by symplectic geometry (an almost symplectic structure is a non-degenerate two-form that is not necessarily closed). } \emph{spans} is constructed in~\cite[Section 4]{gorsk}. 
They use it for an application in legendrian knot theory.
Even though we only deal with left Calabi--Yau structures in this paper, we briefly explain how to construct the $(\infty,n)$-category of iterated (almost) right Calabi--Yau spans in \cref{appendixC}. 

\subsection*{Acknowledgements}

We thank Joost Nuiten for fruitful discussions about twisted arrow categories and duality.
DC warmly thanks Rune Haugseng and Claudia Scheimbauer for the joint work \cite{CHS} (without which the present paper wouldn't exist), 
as well as Alex Takeda for his explanations about \cite{ST} (and more generally for various discussions about Calabi--Yau structures). 
The first and second authors have received funding from the European Research Council (ERC) under the
European Union's Horizon 2020 research and innovation programme (Grant Agreement No.\ 768679). The first author is partially supported by the Agence Nationale de la Recherche (project n°\! ANR-24-CE40-2583-01).

\makeatother 

\makeatletter

\section{Higher categories of iterated (co)spans}


\subsection{Recollection on $(\infty,n)$-categories}

All along the paper we use the \textit{complete $n$-fold Segal spaces} of Barwick \cite{BarwickThese} as models for $(\infty,n)$-categories, and 
we follow the notation and terminology of Haugseng \cite[\S4 and \S7]{Haug}. In particular, we write $\cat{S}$ for the $\infty$-category of small $\infty$-groupoids 
(or spaces), and $\cat{Cat}$ for the $\infty$-category of small $\infty$-categories (or quasi-categories). We consider the following: 
\begin{itemize}
\item The $\infty$-category $\cat{Cat}^n(\cat{S})$ of \textit{$n$-uple Segal spaces}. It is constructed iteratively: 
$\cat{Cat}^n(\mathcal{S}):=\mathrm{Cat}\big(\cat{Cat}^{n-1}(\cat{S})\big)$, where 
$\cat{Cat}(\cat{C})$ is the full sub-$\infty$-category of $\cat{Fun}(\Delta^{\op},\cat{C})$ spanned by Segal objects (here $\cat{C}$ is an $\infty$-category with finite limits). 
\item The full sub-$\infty$-category $\cat{Seg}_n(\cat{S})$ of $\cat{Cat}^n(\cat{S})$ spanned by \textit{$n$-fold Segal spaces}. 
An $n$-uple Segal space $\cat{D}_{\bullet,\dots,\bullet}$ is an $n$-fold Segal space if $\cat{D}_{k_1,\dots,k_m,0,\bullet,\dots,\bullet}$ is constant, for every $0\leq m<n$ 
and all $k_1,\dots,k_m\in\mathbb{Z}_{\geq0}$. 
\item the full sub-$\infty$-category $\cat{CSS}^n(\cat{S})$ of $\cat{Seg}_n(\cat{S})$ spanned by \textit{complete} $n$-fold Segal spaces. 
An $n$-fold Segal space $\cat{D}_{\bullet,\dots,\bullet}$ is complete if each $\cat{D}_{1,\dots,1,\bullet,0,\dots,0}$ is complete as a ($1$-fold) Segal space. 
Recall that a Segal space $X$ is complete if the map $X_0\to X_{eq}$ from objects ($0$-simplices) to equivalences (invertible $1$-simplices) is an equivalence.  
\end{itemize}
On the one hand, the full embedding $\cat{Seg}_n(\cat{S})\hookrightarrow \cat{Cat}^n(\cat{S})$ has a right adjoint, 
called the \textit{$n$-fold underlying Segal space} functor, and denoted $U_{\cat{Seg}}$.  
On the other hand, the full embedding $\cat{CSS}^n(\cat{S})\hookrightarrow \cat{Seg}_n(\cat{S})$ has a left adjoint that preserves products, 
called the \textit{completion} functor. 

\medskip

Whenever $n=1$, we often drop it from the notation: e.g.\ $\cat{CSS}(\cat{S}):=\cat{CSS}^1(\cat{S})$. 
Note that $\cat{CSS}(\cat{S})\simeq\cat{Cat}$.  


\subsection{Iterated (co)spans}\label{subsec-cospans}
 
Let $\cat{C}$ be an $\infty$-category with pull-backs. 
Recall from \cite[Definition 5.16]{Haug} the $n$-fold Segal space $\cat{Span}_n(\cat{C})$ of iterated spans. 
It is defined as the underlying $n$-fold Segal space of the $n$-uple Segal space $\cat{SPAN}_n(\cat{C})$: $\cat{Span}_n(\cat{C}):=U_{\cat{Seg}}\cat{SPAN}_n(\cat{C})$. 
Let us summarize the construction of the $n$-uple Segal space $\cat{SPAN}_n(\cat{C})$ (see \cite[\S5]{Haug} for more details): 
\begin{itemize}
\item For every non negative integer $k$, define $\Sigma^k\subset[k]\times[k]^{\op}$, where $[k]$ is the (totally) ordered set $\{0,\dots,k\}$, 
be the sub-poset spanned by pairs $(i,j)$ such that $i\leq j$. 
This defines a cosimplicial $\infty$-category (in fact, cosimplicial poset) $\Sigma^\bullet: \Delta\to \cat{Cat}$. 
Taking the $n$-fold product, we get a functor $\Sigma^{\bullet,\dots,\bullet}:\Delta^n\to \cat{Cat}$. 
\item Let $\Lambda^k\subset\Sigma^k$ be the sub-poset spanned by pairs $(i,j)$ such that $j\leq i+1$ (i.e.~$j\in\{i,i+1\}$). 
We write $\Lambda^{k_1,\dots,k_n}:=\Lambda^{k_1}\times\cdots\times\Lambda^{k_n}$. 
\item Composing with the functor $\cat{Map}_{\cat{Cat}}(-,\cat{C}):\cat{Cat}\to \cat{S}$ we obtain an $n$-fold simplicial space 
$\overline{\cat{SPAN}}_n(\cat{C}):\Delta^{n,\op}\to \cat{S}$. 
\item A functor $\Sigma^{k_1,\dots,k_n}\to \cat{C}$ is called \emph{cartesian} if it is a right Kan extension of its restriction to $\Lambda^{k_1\dots,k_n}$. 
\item The assignment $\cat{SPAN}_n(\cat{C}):([k_1],\dots,[k_n])\mapsto \cat{SPAN}_n(\cat{C})_{k_1,\dots,k_n}$, where 
$\cat{SPAN}_n(\cat{C})_{k_1,\dots,k_n}$ is the subspace of $\overline{\cat{SPAN}}_n(\cat{C})_{k_1,\dots,k_n}$ spanned by cartesian functors, 
still defines a functor $\Delta^{n,\op}\to \cat{S}$ (this is \cite[Corollary 5.12 and Definition 5.15]{Haug}). 
\item The fact that $\cat{SPAN}_n(\cat{C})$ is an $n$-uple Segal space is proven in \cite[Proposition 5.14 and Definition 5.15]{Haug}. 
\end{itemize}
Finally, the $n$-fold Segal space $\cat{Span}_n(\cat{C})$ is complete (\cite[Corollary 8.5]{Haug}). 
The construction $\cat{Span}_n(-)$ defines a limit preserving functor $\cat{Cat}^{\cat{p.b.}}\to \cat{CSS}^n(\cat{S})$, 
where $\cat{Cat}^{\cat{p.b.}}\subset\cat{Cat}$ is the sub-$\infty$-category consisting of $\infty$-categories with pull-backs 
and pull-back preserving functors between them (see \cite[Theorem 4.1]{LB}). 

\begin{notation}
For $\cat{D}=\cat{C}^{\op}$ (therefore $\cat{D}$ is assumed to have \underline{push-outs}), we write 
\[
\overline{\cat{COSPAN}}_n(\cat{D}):=\overline{\cat{SPAN}}_n(\cat{C})\,,~
\cat{COSPAN}_n(\cat{D}):=\cat{COSPAN}_n(\cat{C})~\textrm{and}~\cat{coSpan}_n(\cat{D}):=\cat{Span}_n(\cat{C})\,.
\]
Hence $\cat{coSpan}_n(-)$ defines a limit preserving functor $\cat{Cat}^{\cat{p.o.}}\to \cat{CSS}^n(\cat{S})$, 
where $\cat{Cat}^{\cat{p.o.}}\subset\cat{Cat}$ is the sub-$\infty$-category consisting of $\infty$-categories with push-outs and push-out preserving functors 
between them
\end{notation}


\subsection{Our main examples}\label{ssec-our-examples}

Let $\cat{D}$ be an $\infty$-categories with push-outs and $d$ an object of $\cat{D}$. 
We consider the comma category functor 
\[
d\!\downarrow \!-:=-\underset{\cat D}{\times}{}^{d/}\!\cat{D}\,:\,\left(\cat{Cat}_{/\!\!/\cat{D}}\right)^{\mathrm{p.o.}}\,\longrightarrow\,\cat{Cat}^{\mathrm{p.o.}}\,.
\]
If we have a natural transformation $F\Rightarrow GH$, where $F:\cat{A}\to\cat{D}$, $G:\cat{B}\to\cat{D}$ are functors and $H:\cat{A}\to\cat{B}$ is 
push-out preserving, then we abuse notation and still denote by $H$ the push-out preserving functor $d\!\downarrow \!F\to d\!\downarrow \!G$ given by this natural transformation. 


All examples in this paper can be described using $\cat{D}=\scat{Mod}_k$ and $d=k[s]$: 
\begin{itemize}
\item Let us first consider the singular homology functor $\mathrm{H}_\bullet:\cat{S}\to\scat{Mod}_k$. This allows us to define the category of pre-Poincaré $s$-spaces: 
\[
\scat{prePoin}^s:=k[s]\!\downarrow\! \mathrm H_\bullet\,.
\]
\item We then have the functor $\cA^{2,\mathrm{cl}}[2]:\scat{dSt}_k^{\mathrm{op}}\to\scat{Mod}_k$, sending a $k$-derived stack $X$ to its complex 
$\cA^{2,\mathrm{cl}}(X)[2]$ of closed $2$-forms (that is a shift of the realization of the weight $\geq2$ part of the de Rham graded mixed complex $\scat{DR}(X)$), 
leading to the (opposite) category of $(2-s)$-shifted presymplectic stacks: 
\[
\scat{preSymp}^{2-s,\mathrm{op}}:=k[s]\!\downarrow\! \mathrm \cA^{2,\mathrm{cl}}[2]\,.
\]
If one uses the complex of $2$-forms $\cA^{2}[2]$ instead of closed $2$-forms one gets the (opposite) category of $(2-s)$-shifted almost presymplectic stacks 
\[
\scat{alpreSymp}^{2-s,\mathrm{op}}:=k[s]\!\downarrow\! \mathrm \cA^{2}[2]\,.
\] 
\item Finally, we consider the negative cyclic homology functor $\cat{HC}^-:\scat{dgCat}_k\to \scat{Mod}_k$, defined on $k$-linear dg categories, and recall the category of $s$-pre-Calabi--Yau categories
\[
\scat{preCY}^s:=k[s]\!\downarrow\!\cat{HC}^-\,.
\]
Again, if one uses the Hochschild homology functor $\cat{HH}:\scat{dgCat}_k\to \scat{Mod}_k$ instead of the negative cyclic one, one gets  the category of almost $s$-pre-Calabi--Yau categories
\[
\scat{alpreCY}^s:=k[s]\!\downarrow\!\cat{HH}\,.
\]
\end{itemize}	

We have functors between these categories: 
\begin{itemize}
\item Thanks to~\cite[Proposition 5.3]{BD1} there exists a functor $\scat{dgSing}:\cat{S}\to\scat{Cat}_k$ (it e.g.~sends a connected space to its dg algebra of 
chains on the based loop space). By \emph{loc.cit.} there is also a natural transformation $\mathrm{H}_\bullet\Rightarrow \cat{HC}^-\circ \scat{dgSing}$. Hence 
we get a functor
\[
\scat{dgSing}\,:\,\scat{prePoin}^s=k[s]\!\downarrow\! \mathrm H_\bullet\,\longrightarrow\, k[s]\!\downarrow\!\cat{HC}^- =\scat{preCY}^s\,.
\]
\item Recall from~\cite[\S6.1.2]{BCS} the natural transformation 
$\scat{HH}\Rightarrow|\scat{DR}\circ\scat{Perf}|^r$ 
between functors from $\scat{dgCat}_k$ to mixed complexes $\scat{Mod}_k^\delta$. By adjunction, see~\cite[\S 3.1.3]{BCS}, we get 
$\scat{HH}[v]\Rightarrow\scat{DR}\circ\scat{Perf}$, 
where $v$ is of degree $2$ and weight $1$. One can then apply realization of the weight $\geq2$ part, that has a natural tranformation to the weight $2$ part, and get a commuting square of natural transformations 
\[
\begin{tikzcd}
\cat{HC}^- \ar[r,Rightarrow] \ar[d,Rightarrow] &  {\cA^{2,\mathrm{cl}}[2]\circ\scat{Perf}} \ar[d,Rightarrow] \\
\cat{HH} \ar[r,Rightarrow]  &  {\cA^{2}[2]\circ\scat{Perf}}\,.
\end{tikzcd}
\]
Hence we get a commuting diagram of categories 
\begin{equation}\label{eq:diag-of-cat}
\begin{tikzcd}
\scat{preCY}^s \ar[rr,"\scat{Perf}" above] \ar[d] &  &{\scat{preSymp}^{2-s,\mathrm{op}}} \ar[d] \\
{\scat{alpreCY}^s} \ar[rr,"\scat{Perf}" above] &  & {\scat{alpreSymp}^{2-s,\mathrm{op}}} \,.
\end{tikzcd}
\end{equation}
\end{itemize}
Applying $\cat{coSpan}_n(-)$, this yields a commuting diagram of $(\infty,n)$-categories 
\begin{equation}\label{eq:span-diag}
\begin{tikzcd}
\cat{coSpan}_n(\scat{preCY}^s) \ar[rr,"\scat{Perf}" above] \ar[d] &  &{\cat{Span}_n(\scat{preSymp}^{2-s})} \ar[d] \\
\cat{coSpan}_n(\scat{alpreCY}^s) \ar[rr,"\scat{Perf}" above] &  & {\cat{Span}_n(\scat{alpreSymp}^{2-s})} \,.
\end{tikzcd}
\end{equation}


\subsection{Symmetric monoidal structures on iterated (co)spans}\label{ssec-sym-mon-span}

Recall that a \emph{symmetric monoidal $(\infty,n)$-category} is a sequence $(\cat{C}_m,c_m,\varphi_m)_{m\geq0}$ where: 
\begin{itemize}
\item $\cat{C}_m$ is an $(\infty,n+m)$-category,
\item $c_m$ is an object of $\cat{C}_m$, 
\item $\varphi_m:\cat{C}_m\to \cat{C}_{m+1}(c_{m+1},c_{m+1})$ is an equivalence, where  $\cat{C}_{m+1}(c_{m+1},c_{m+1})$ is 
given by the $n+m$-fold Segal space 
\[
\{c_{m+1}\}\underset{(C_{m+1})_{0,\bullet,\dots,\bullet}}{\times}(C_{m+1})_{1,\bullet,\dots,\bullet}\underset{(C_{m+1})_{0,\bullet,\dots,\bullet}}{\times}\{c_{m+1}\}
\]
\end{itemize}
The underlying $(\infty,n)$-category is $\cat{C}_0$, and we say that the above data defines a \emph{symmetric monoidal structure} on $\cat{C}_0$. 
We often abuse the terminology and say that $\cat{C}_0$ is a symmetric monoidal $(\infty,n)$-category. 

\medskip

A source of symmetric monoidal categories is given by the following
\begin{construction}[see {\cite[Remark A.2.21]{CHS}}]\label{constuction-sym}
Let $(\cat{D}_m,d_m,\psi_m)_{m\geq0}$ be a sequence with $\cat{D}_m$ a category, $d_m$ an object of $\cat{D}_m$ and 
$\psi_m:\cat{D}_m\to (\cat{D}_{m+1})_{/d_{m+1},d_{m+1}}$ an equivalence (here $(\cat{D}_{m+1})_{/d_{m+1},d_{m+1}}$ is the double slice category) 
such that 
\[
\psi(d_m)\simeq (d_{m+1}\leftarrow d_{m+1}\rightarrow d_{m+1})\,.
\]
Then we have a symmetric monoidal $(\infty,n)$-category $(\cat{C}_m,c_m,\varphi_m)_{m\geq0}$, where $\cat{C}_m:=\cat{Span}_{n+m}(\cat{D}_m)$, 
$c_m:=d_m$, and $\varphi$ is given by 
\[
\cat{Span}_{n+m}(\cat{D}_m)\overset{\psi}{\longrightarrow}\cat{Span}_{n+m}\big( (\cat{D}_{m+1})_{d_{m+1},d_{m+1}}\big)\simeq
\cat{Span}_{n+m+1} (\cat{D}_{m+1})(d_{m+1},d_{m+1})\,.
\]
This construction is functorial. There is also a version for cospans, using the double \underline{co}slice category. 
\end{construction}

Assume that $\cat{C}$ is a category having push-outs and an initial object $\emptyset$, and $F:\cat{C}\to\scat{Mod}_k$ is a functor such that $F(\emptyset)\simeq0$. 
Consider $\cat{D}_m:=k[s-m]\!\downarrow\!F$ and $d_m:=\big(\emptyset,k[s-m]\to 0\big)$. We then have an equivalence $\psi_m$ as in \cref{constuction-sym} 
(cospan/double coslice version), which follows from the equivalences $\cat{C}\simeq {}^{\emptyset,\emptyset/\!}\cat{C}$, 
$\scat{Mod}_k\simeq {}^{0,0/\!}\scat{Mod}_k$, and ${}^{k[p+1]/\!}\scat{Mod}_k\simeq {}^{k[p]\to0,k[p]\to0/\!}({}^{k[p]/\!}\scat{Mod}_k)$. 
Applying \cref{constuction-sym}, we get a symmetric monoidal structure on $\cat{coSpan}_n(k[s]\!\downarrow\!F)$. 

\begin{example}
All the examples from \cref{ssec-our-examples} fit into this framework: indeed, $\cat{S}$, $\scat{dgCat}_k$ and $\scat{dSt}_k^{\mathrm{op}}$ have an initial object. 
Using the construction and its functoriality (note that $\scat{Perf}(\emptyset)=\mathrm{pt}$), we get that the diagram \eqref{eq:span-diag} can be promoted to a 
diagram of \emph{symmetric monoidal} $(\infty,n)$-categories. 
\end{example}


 

 
 
\begin{remark}
It follows from \cite{Haug} that every object of $\cat{coSpan}_n(k[s]\!\downarrow\!F)$ is fully dualizable in the sense of \cite{Lurie-TFT}. 
This means in particular that every object leads to a $\cat{coSpan}_n(k[s]\!\downarrow\!F)$-valued $n$-dimensional fully extended 
\underline{framed} TFT. In comparison, the main difficulties in \cite{CHS} (of which the present paper gives a noncommutative version) are:  
\begin{enumerate}
\item To upgrade the above to an \emph{oriented} fully extended TFT, in the case of $\scat{preSymp}^s$. One of the important innovations of \cite{CHS} is to prove 
that there is an oriented fully extended TFT taking values in (pre)Poincaré spaces. We will crucially use it: all our TFTs will factor through the higher 
symmetric monoidal category of (pre)Poincaré spaces. 
\item To show that if the shifted pre-symplectic structure is non-degenerate (i.e.~it is shifted symplectic) then the TFT takes values in the symmetric 
monoidal sub-$(\infty,n)$-category of lagrangian correspondences. In the present paper we rework the treatment from \cite{CHS} for dealing with 
non-degeneracy, making it more general, so that it can be applied uniformly (and functorially!) to several different contexts. 
\end{enumerate}

\end{remark}

\makeatother

\makeatletter

\section{Higher categories of non-degenerate (co)spans} 

Let $(\cat{C},\mathcal G,\Delta,\mathcal F)$ be a $4$-tuple where $\cat{C}$ is a category, $\mathcal G:\cat{C}\to\cat{Cat}$ is a functor, 
$\Delta:\cat{C}\to \int\mathcal G$ is a section of the fibration $\int\mathcal G\to\cat{C}$, and 
$\mathcal F:\cat{C}\to \int\cat{Fun}\big(\mathcal G(-),\scat{Mod}_k\big)$ 
is a section of $\int\cat{Fun}\big(\mathcal G(-),\scat{Mod}_k\big)\to\cat{C}$. Finally, let $\cat{C}'\subset\cat{C}$ be a full subcategory such that, for every object 
$c\in\cat{C}'$, $|\mathcal F(c)[-s]|:\mathcal G(c)\to\mathcal S$ is corepresentable. 

\medskip

In the first part of this section we construct a natural functor 
\[
\flat_s\,:\,k[s]\!\downarrow\!( \mathcal F\Delta)_{|\cat{C}'}\,\longrightarrow\,\int\mathrm{Tw}^\ell \mathcal G_{|\cat{C}'}\subset \int\mathrm{Tw}^\ell \mathcal G\,,
\]
where $\mathrm{Tw}^\ell$ is the twisted arrow category construction. On objects, the functor $\flat_s$ sends a pair $(c,f)$, 
where $c\in\cat{C}'$ and $f:k[s]\to \mathcal F(c)\big(\Delta(c)\big)$ to $(c,g,\Delta(c),x)$, where $g\in \mathcal G(c)$ corepresents $|\mathcal F(c)[-s]|$ 
and $x\in\cat{Hom}_{\mathcal G(c)}\big(g,\Delta(c)\big)\simeq|\mathcal F(c)\big(\Delta(c)\big)[-s]|$ corresponds to $f$. 

\medskip

We then use the functor $\flat_s$ in the second part of the section in order to define non-degenerate diagrams in $\cat{C}'$, as well as the (non full) 
sub-$(\infty,n)$-category $\cat{coSpan}_n(\cat{C}')^{\mathrm{ndg}}\subset \cat{coSpan}_n(\cat{C}')$ of non-degenerate cospans, whenever $\cat{C}'$ 
admits push-outs. We also discuss symmetric monoidal structures and full dualizability. 



\subsection{Some abstract nonsense}\label{ssec-abstract-nonsense}

Recall the left Kan extension $\cat{Fun}(\mathcal D_1,\mathcal S)\to \cat{Fun}(\mathcal D_2,\mathcal S)$, $f\mapsto\cat{Lan}_\varphi(f)$ along any functor 
$\varphi:\mathcal D_1\to\mathcal D_2$, which yields a functor
\begin{eqnarray*}
\cat{Ev}\,:\,\cat{Cat} & \longrightarrow & \cat{Cat}_{\!/\!\!/\mathcal S} \\
\mathcal D & \longmapsto & \big(\cat{ev}:\cat{Fun}(\mathcal D,\mathcal S)\times\mathcal D\to\mathcal S\big)
\end{eqnarray*}
as we have natural transformations $f\Rightarrow\cat{Lan}_\varphi(f)\varphi$.
We then denote by $\mathcal T$ and $\mathcal U$ the compositions 
\[
\cat{Cat}\overset{\cat{Ev}}{\longrightarrow}\cat{Cat}_{\!/\!\!/\mathcal S}\overset{\mathrm{forget}}{\longrightarrow}\cat{Cat}
\quad\text{and}\quad
\cat{Cat}\overset{\cat{Ev}}{\longrightarrow}\cat{Cat}_{\!/\!\!/\mathcal S}\overset{\int}{\longrightarrow}\cat{Cat}
\]
respectively. Here we write ``$\mathrm{forget}$'' instead of ``$s$'' (that is used in \cref{lemmpullpull}) in order to avoid any confusion with the shift by $s$ in $k[s]$. 

\medskip

We have a sequence of functors $\int \mathcal T\mathcal G\to \int\mathcal T\to\int \mathrm{forget}=\mathcal X\to \mathcal S$, 
whose composition we denote by $\mathrm{ev}_{\mathcal G}$. On objects, it sends 
$\big(c\in\mathcal C,F:\mathcal G(c)\to\cS,x\in\mathcal G(c)\big)$ to $F(x)$. 
Note that $|\mathcal F[-s]|\times\Delta$ is a section of $ \int \mathcal T\mathcal G\to\mathcal C$, thanks to the natural equivalence 
\[
\int\! \mathrm{Fun}(\mathcal G(-),\mathcal S) \times_\cC  \int\!\mathcal G
\simeq \int\!\mathrm{Fun}(\mathcal G(-),\mathcal S)   \times   \mathcal G
=  \int \mathcal T\mathcal G\,.
\]

Observe that in the construction of $\mathcal T$ and $\mathrm{ev}_{\mathcal G}$ we can replace $\cS$ with any other category, such as $\scat{Mod}_k$ for instance, 
in which case we write $\mathcal T_k$ and $\mathrm{ev}_{\mathcal G,k}$. 
Doing so, and using \cref{lemmgrothcosl}, we get that 
\begin{equation}\label{eq:int-comma}
\int\mathrm{ev}_{\mathcal G}(|\mathcal F[-s]|\times\Delta)
\simeq k\!\downarrow\!\big(\mathrm{ev}_{\mathcal G,k}(\mathcal F[-s]\times\Delta)\big)
\simeq k[s]\!\downarrow\!\big(\mathrm{ev}_{\mathcal G}(\mathcal F\times\Delta)\big)\,.
\end{equation}

To lighten the notation, we allow ourselves to write $\mathcal F\Delta$ instead of $\mathrm{ev}_{\mathcal G,k}(\mathcal F\times\Delta)$ 
(we actually already did it in the introduction to this section).

\begin{proposition} \label{lemma:twisted}
There is a functor
\[
\flat_s:k[s]\!\downarrow\!\mathcal F\Delta\,\longrightarrow\, \int \! \mathcal U\mathcal G
\]
over $\cC$. 
\end{proposition}
\begin{proof}
We first prove the following: 
\begin{lemma}\label{lemma:pullback}
There is an equivalence 
\[
\int \! \mathcal U\mathcal G \simeq\left(\int \! \mathcal T\mathcal G\right)\underset{\mathcal S}{\times}{}^*\mathcal S
\]
over $\cC$. In other words, $\int \! \mathcal U\mathcal G\simeq \int \!\mathrm{ev}_{\mathcal G}$. 
\end{lemma}
\begin{proof}[Proof of the lemma]
There are pullback squares
\[
\xymatrix{
\int  \mathcal T\mathcal G\ar[r]\ar[d]&\mathcal X\ar[r]\ar[d]&{}^*\cat{Cat}\ar[d]\\
\mathcal C\ar[r]^-{\cat{Ev}\mathcal G}&\cat{Cat}_{\!/\!\!/\mathcal S}\ar[r]&\cat{Cat}} 
\quad\raisebox{-.6cm}{\text{ and }}\quad
\xymatrix{\int  \mathcal U\mathcal G\ar[r]\ar[d]&\mathcal Y\ar[r]\ar[d]&{}^*\cat{Cat}\ar[d]\\
\mathcal C\ar[r]^-{\cat{Ev}\mathcal G}&\cat{Cat}_{\!/\!\!/\mathcal S}\ar[r]^-\int&\cat{Cat}} 
\]
which yield, thanks to \cref{lemmpullpull}, 
\[
\int \! \mathcal U\mathcal G
=\mathcal C\underset{\cat{Cat}_{\!/\!\!/\mathcal S}}{\times}\mathcal Y
\simeq\mathcal C\underset{\cat{Cat}_{\!/\!\!/\mathcal S}}{\times}\mathcal X\underset{\mathcal S}{\times}{}^*\mathcal S
=\left(\int \! \mathcal T\mathcal G\right)\underset{\mathcal S}{\times}{}^*\mathcal S\,.\qedhere
\]
\end{proof}
Summarizing, we have a functor 
$\int \!\mathrm{ev}_{\mathcal G} (|\mathcal F[-s]|\times \Delta)\to \int\! \mathrm{ev}_{\mathcal G}\simeq\int\!\mathcal U\mathcal G$. 
We conclude using the equivalence \eqref{eq:int-comma}. 
\end{proof}

Observe that objects of $\int\mathcal U\mathcal G$ are $4$-tuples $(a,H,m,x)$ with $a\in\cC$, $H:\mathcal G(a)\to\cS$, $m\in\mathcal G(a)$ and $x\in H(m)$. 
On objects, 
\[
\flat_s\,:\big(a,\alpha:k[s]\to \mathcal F(a)\Delta(a)\big)\,\longmapsto\,\big(a,|\mathcal F(a)[-s]|,\Delta(a),\alpha[-s]\big)\,,
\]
where $\alpha[-s]:k\to \big(\mathcal F(c)\Delta(c)\big)[-s]$ determines an element in $|\mathcal F(c)[-s]|(\Delta(c))$. 
When it doesn't cause any confusion, we lighten the notation and write $\alpha$ (instead of $\alpha[-s]$) for this element as well. 

\begin{examples}\label{examples-universal}
We now explain how the map $\flat_s:k[s]\!\downarrow\!\mathcal F\Delta\,\longrightarrow\, \int \! \mathcal U\mathcal G$ 
behaves on objects for the three examples from \cref{ssec-our-examples}. 
\begin{enumerate}[label=(\roman*)]
\item $\mathcal C=\cS$, $\mathcal G=\mathbf{Mod}_{\mathcal L}$ with $\mathcal L:=\scat{dgSing}$, $\Delta:Y\mapsto k_Y$ and 
$\mathcal F:Y\mapsto\mathrm{H}_\bullet(Y,-)$ is given by singular homology with coefficients. 
On objects 
\begin{eqnarray*}
\flat_s\,:\,\scat{prePoin}^s&\longrightarrow & \int\mathcal U\scat{Mod}_{\mathcal L}\\
\Big(Y,c:k[s]\to \mathrm{H}_\bullet(Y)\Big) & \longmapsto & \Big(Y, \mathrm{H}_s(Y,-),k_Y, c\in \mathrm{H}_s(Y,k_Y)\Big)\,,
\end{eqnarray*}
where $\mathrm{H}_s(Y,-):=|\mathrm{H}_\bullet(Y,-)[-s]|$. 

\item  $\mathcal C=\scat{dgCat}_k$, $\mathcal G=\scat{BiMod}$ with 
\[
(\cA\to \cB)\longmapsto \big(\scat{BiMod}_\cA\overset{\otimes_{\cA^\mathrm e}\cB^\mathrm e}{\longrightarrow}\scat{BiMod}_\cB\big)\,,
\]
$\Delta:\cA\mapsto \cA\in \scat{BiMod}_\cA$  and $\mathcal F:\cA\mapsto \cat{HH}(\cA,-)$. 
On objects 
\begin{eqnarray*}
\flat_s\,:\, \scat{alpreCY}^s & \longrightarrow & \int\mathcal U\scat{BiMod} \\
\Big(\cA,c:k[s]\to\cat{HH}(\cA)\Big) & \longmapsto & \Big(\cA,\cat{HH}_s(\cA,-),\cA,c\in\cat{HH}_s(\cA,\cA)\Big)\,, 
\end{eqnarray*}
where $\cat{HH}_s(\cA,-):=|\cat{HH}(\cA,-)[-s]|$.

\item $\mathcal C=\mathbf{dSt}_k^{\mathbb{L},\mathrm{op}}$ where $\mathbf{dSt}_k^{\mathbb{L}}\subset\mathbf{dSt}_k$ is the full subcategory of 
derived stacks that admit a cotangent complex, $\mathcal G=\scat{QCoh}$ with
\[
(f:X\to Y)\longmapsto\big(\scat{QCoh}(Y)\overset{f^*}{\longrightarrow}\scat{QCoh}(X)\big)\,,
\]
$\Delta:X\mapsto\mathbb L_X$ and $\mathcal F:X\mapsto\Gamma(\mathbb L_X\otimes-)$. 
On objects
\begin{eqnarray*}
\flat_s\,:\, k[s]\!\downarrow\!\Gamma\big((\mathbb{L}_{-})^{\otimes2}[2]\big) & \longrightarrow & \int\mathcal U\scat{QCoh}\\
\Big(X,\omega:k[s]\to\Gamma(\mathbb{L}_X \otimes \mathbb{L}_X)[2]\Big) & \longmapsto &  
\Big(X,|\Gamma(\mathbb{L}_X \otimes -)[2-s]|,\mathbb{L}_X,\omega\in|\Gamma(\mathbb{L}_X^{\otimes2})[2-s]|\Big)\,.
\end{eqnarray*}
\end{enumerate}
\end{examples}
 
\begin{remark}\label{remark-symmetrization}
It follows from \cite[theorem 2.6]{CSa} that the natural transformation 
\[
\Gamma\big(\mathrm{Sym}^2(\mathbb{L}_{-}[-1])[2]\big)\Longrightarrow \mathcal A^{2}_{|\scat{dSt}_k^{\mathrm{perf}\mathbb{L},\mathrm{op}}}\,,
\] 
between functors $\scat{dSt}_k^{\mathbb{L},\mathrm{op}}\to\scat{Mod}_k$, is an equivalence. 
Hence the symmetrization of tensors defines a natural transformation 
\[
\mathcal A^{2}_{|\scat{dSt}_k^{\mathbb{L},\mathrm{op}}}\,\Longrightarrow\,\Gamma\big((\mathbb{L}_{-})^{\otimes2}\big)
\]
that induces a functor 
\[
\scat{alpreSymp}^{\mathbb{L},s,\mathrm{op}}\,\longrightarrow\,
k[s]\!\downarrow\!\Gamma\big((\mathbb{L}_{-})^{\otimes2}[2]\big)
\]
$\scat{alpreSymp}^{\mathbb{L},s}\subset \scat{alpreSymp}^{s}$ is the full subcategory spanned by stacks that admit a cotangent complex. 
\end{remark}
 

\subsection{Twisted arrows and functoriality}

\begin{definition}
The twisted arrow category $\mathrm{Tw}^\ell(\mathcal D)$ associated to $\mathcal D\in\cat{Cat}$ is the pullback of 
$\mathcal U(\mathcal D)\to\mathrm{Fun}(\mathcal D,\mathcal S)\times\mathcal D=\mathcal T(\mathcal D)$ along the Yoneda embedding 
\[
\mathcal Y_{\mathcal D}:\mathcal D^\mathrm{op}\times\mathcal D\to\mathrm{Fun}(\mathcal D,\mathcal S)\times\mathcal D\,.
\]
In other words, $\mathrm{Tw}^\ell(\mathcal D)$ is the full subcategory of $\mathcal U(\mathcal D)$ with objects $(F,d,x\in F(d))$ such that $F$ is corepresentable.
\end{definition}
In the above definition everything is functorial in $\mathcal D\in\cat{Cat}$, hence we have a functor 
$\mathrm{Tw}^\ell:\cat{Cat}\to\cat{Cat}$ that can be formally defined as $\big((-)^\mathrm{op}\times\mathrm{id}\big)\times_{\mathcal T}\mathcal U$. 
As a result, using \cref{lemma:pullback} we have 
\[
\int\mathrm{Tw}^\ell\mathcal G
\simeq\int\big((-)^\mathrm{op}\mathcal G\times\mathcal G\big)\times_{\int\mathcal T\mathcal G}\int\mathcal U\mathcal G
\simeq \int\big((-)^\mathrm{op}\mathcal G\times\mathcal G\big)\times_{\cS}{}^*\mathcal S\,,
\]
meaning that $\int\mathrm{Tw}^\ell\mathcal G=\int\Map_{\mathcal G}$, where 
$\Map_{\mathcal G}=\mathrm{ev}_{\mathcal G}\mathcal Y_{\mathcal G}:\int\big((-)^\mathrm{op}\mathcal G\times\mathcal G\big)\to\cS$ 
sends a triple $(c,g_1,g_2)$, where $c\in\cC$ and $g_1,g_2\in\mathcal G(c)$, to $\Map(g_1,g_2)$. 

\medskip

Recall that $\cC'\subset \cC$ is a full subcategory such that, for every object $c\in\cC'$, the functor $|\mathcal F(c)[-s]|$ is corepresentable. 
This implies that, when restricted to $\cC'\subset \cC$, the functor $\flat_s$ from \cref{lemma:twisted} factors as 
\[
k[s]\!\downarrow\!( \mathcal F\Delta)_{|\cat{C}'}\,\longrightarrow\,\int\mathrm{Tw}^\ell \mathcal G_{|\cat{C}'}
\subset \int\mathcal U\mathcal G_{|\cat{C}'}\,.
\] 
Indeed, when restricted to $\cC'$, the section $|\mathcal F[-s]|:\cC\to \int\Fun(\mathcal \mathcal G(-),\cS)$ factor through 
a section $f_s:\cC'\to\int\big((-)^{\mathrm{op}}\mathcal G\big)$; therefore $\flat_s$ is the composition
\[
k[s]\!\downarrow\!( \mathcal F\Delta)_{|\cat{C}'}
\simeq\int\mathrm{ev}_{\mathcal G_{|\cC'}}(|\mathcal F[-s]|\times\Delta)
\simeq\int\Map_{\mathcal G_{|\cC'}}(f_s\times\Delta)
\longrightarrow\int\Map_{\mathcal G_{|\cC'}}\simeq\int\mathrm{Tw}^{\ell}\mathcal G_{|\cC'}\,.
\]
In other words, $\flat_s$ is the pullback of $f_s\times\Delta$ along $\int\mathrm{Tw}^\ell\mathcal G\to \int\big((-)^\mathrm{op}\mathcal G\times\mathcal G\big)$. 
\begin{examples}\label{examples-twisted}
Let us go back to \cref{examples-universal}: 
\begin{enumerate}[label=(\roman*)]
\item Let $\cC'=\cS^{\mathrm{ft}}$ be the full subcategory of spaces of finite type. Following~\cite[\S5.1.1]{BD1}, if $Y$ is of finite type then the functor 
$\mathrm{H}_s(Y,-)$ is corepresentable by $\zeta_Y[s]$. Hence a class $c\in\mathrm{H}_s(Y)=\mathrm{H}_s(Y,k_Y)$ determines a map $c^\flat:\zeta_Y[s]\to k_Y$, 
and thus on objects 
\begin{eqnarray*}
\flat_s\,:\,\scat{prePoin}^{\mathrm{ft},s} &\longrightarrow & \int\mathrm{Tw}^\ell\scat{Mod}_{\mathcal L}\\
\Big(Y,c:k[s]\to \mathrm{H}_\bullet(Y)\Big) & \longmapsto & \Big(Y, \zeta_Y[s],k_Y, c^\flat\Big)\,.
\end{eqnarray*}

\item Let $\cC'=\scat{dgCat}_k^{\mathrm{sm}}$ be the full subcategory of smooth dg-categories. 
By definition, if $\cA$ is smooth then the functor $\mathrm{HH}_s(\cA,-)$ is corepresentable by $\cA^\vee[s]$. 
Hence a class $c\in\mathrm{HH}_s(\cA,\cA)=\mathrm{HH}_s(\cA)$ determines a map $c^\flat:\cA^\vee[s]\to\cA$, and thus on objects
\begin{eqnarray*}
\flat_s\,:\, \scat{alpreCY}^{\mathrm{sm},s} & \longrightarrow & \int\mathrm{Tw}^\ell\scat{BiMod} \\
\Big(\cA,c:k[s]\to \mathrm{HH}(\cA)\Big) & \longmapsto & \Big(\cA, \cA^\vee[s],\cA, c^\flat\Big)\,.
\end{eqnarray*}

\item Let $\cC'=\mathbf{dSt}_k^{\mathrm{perf}\mathbb{L},\mathrm{op}}$ be the full subcategory of derived stacks with a perfect cotangent complex. 
For such a stack $X$, we write $\mathbb{T}_X$ for the dual of the cotangent complex $\mathbb{L}_X$; its shift $\mathbb{T}_X[2-s]$ corepresents 
$|\Gamma(\mathbb{L}_X\otimes -)[s-2]|$, so that any section $\omega\in|\Gamma(\mathbb{L}_X^{\otimes2})[s-2]|$ determines a map 
$\omega^\flat:\mathbb{T}_X[2-s]\to\mathbb{L}_X$. Therefore, on objects
\begin{eqnarray*}
\flat_s\,:\, k[s]\!\downarrow\!\Gamma\big((\mathbb{L}_{-})^{\otimes2}[2]\big)_{|\cC'} & \longrightarrow & \int\mathrm{Tw}^\ell\scat{QCoh}\\
\Big(X,\omega:k[s]\to\Gamma(\mathbb{L}_X^{\otimes2})[2]\Big) & \longmapsto & \Big(X, \mathbb{T}_X[2-s],\mathbb{L}_X, \omega^\flat\Big)\,.
\end{eqnarray*}
\end{enumerate}
\end{examples}

\begin{remark}\label{remark-enriched}
In all our examples of interest, $\mathcal G$ in fact takes values in $\scat{dgCat}_k$, the functor $\mathcal F(a):\mathcal G(a)\to\scat{Mod}_k$ is a dg-functor, and if 
$a\in\cC'$ then $F(a)$ is corepresentable in the enriched sense: there exists an object $f(a)$ such that $\mathcal F(a)\simeq\underline{\Map}\big(f(a),-\big)$, and thus 
$|\mathcal F(a)[-s]|$ is corepresented by $f_s(a)=f(a)[s]$. 
\end{remark}

\paragraph{Functoriality}
For clarity of exposition we now assume that $|\mathcal F(c)[-s]|$ is corepresentable (by $f_s(c)\in\mathcal G(c)$) for every $c\in\cC$ 
(this is harmless as we can always replace $\cC$ with $\cC'$). Let us explain in which sense the assignment 
\[
(\cC,\mathcal G,f_s,\Delta)\,\longmapsto\,\big(\flat_s:k[s]\!\downarrow\!\mathcal F\Delta\,\to\,\int\! \mathrm{Tw}^\ell \mathcal G\big)
\]
is functorial. If $\rho: \cC \to \tilde\cC$ is a functor and $\sigma:  \mathcal{G}  \Rightarrow   \tilde{\mathcal{G}}  \rho$ is a natural transformation, then 
$\rho$ lifts to a morphism of fibrations $\rho^\sigma:\int\mathcal{G}\to\int\tilde{\mathcal G}$. 

\begin{lemma} \label{lemma:functoriality2} 
Let $f_s,\Delta:\cC\to\int\mathcal G$ and $\tilde{f_s},\tilde\Delta:\tilde\cC\to\int\tilde{\mathcal G}$ be sections that make the square 
\begin{equation}\label{hypface}
\xymatrix{
\mathcal C\ar[r]^-{f_s\times\Delta}\ar[d]&\int\!\big((-)^{\mathrm{op}}\mathcal G\times\mathcal G\big)\ar[d]^{\rho^\sigma}\\
\tilde{\mathcal C}\ar[r]^-{\tilde{f_s}\times\tilde\Delta}&\int\!\big((-)^{\mathrm{op}}\tilde{\mathcal G}\times\tilde{\mathcal G}\big)
}
\end{equation}
commute. Then there exists a commuting diagram
  \[
\xymatrix{
		\int  \Map_{\mathcal G}(f_s\times \Delta)\ar[r]^-{\flat_s}\ar[d]&\int \mathrm{Tw}^\ell \mathcal G\ar[d]\\
		\int  \Map_{\tilde{\mathcal G}}(\tilde{f_s}\times \tilde\Delta)\ar[r]^-{\flat_s}&\int \mathrm{Tw}^\ell \tilde{\mathcal G}}\,.
\]
\end{lemma} 

\begin{proof} 
Note that the square \eqref{hypface} commutes in $\cat{Cat}_{\!/\!\!/\cS}$. Indeed, the functor $\rho^\sigma$ lifts to $\cat{Cat}_{\!/\!\!/\cS}$, 
as $\sigma$ itself induces a natural transformation $\Map_{\mathcal G}\Rightarrow\Map_{\tilde{\mathcal G}}$. 
The commuting square we want is obtained by applying $\int=\times_{\cS}{}^*\cS$. 
 \end{proof} 


\subsection{Examples}\label{subsection-examples}

We clarify the functoriality properties of $\flat_s$ applied to our examples of interest.


\subsubsection{Singular chains \& local systems}\label{subsing}

Here we rely on the results and terminology from~\cite[\S5.1.1]{BD1}, and borrow the notation from above: 
\begin{itemize} 
\item $\cC=\cS^\mathrm{ft}$ is the full subcategory of spaces of finite type in $\cS$, $\tilde\cC=\scat{dgCat}_k^\mathrm{ft}$ is the full subcategory 
of finite type dg-categories in $\scat{dgCat}_k^\mathrm{sm}$, and $\rho=\cLoc=\scat{dgSing}$. 
\item $\mathcal G=\scat{Mod}_\cLoc$ and $\tilde{\mathcal G}=\scat{BiMod}$ (see \cref{examples-universal}). 
Recall that $\scat{Mod}_{\cLoc(Y)}$ is equivalent the dg-category of ($\infty$-)local systems (of cochain complexes of $k$-modules) on $Y$. 
\item For every $Y\in \cS_k^\mathrm{ft}$ there exists a $\mathcal L(Y)-\mathcal L(Y)^e$-bimodule $M_Y$ that induces a functor 
\[
-\otimes_{\mathcal L(Y)}M_Y:\scat{Mod}_{\mathcal L(Y)}\to\scat{BiMod}_{\mathcal L(Y)}
\]
and actually a natural transformation $\sigma:\mathcal G=\scat{Mod}_{\mathcal L}\Rightarrow\scat{BiMod}_{\mathcal L}=\tilde{\mathcal G}\rho$. 
\item Following \cref{examples-universal}, $\Delta(Y)=k_Y$, $\tilde\Delta(\cA)=\cA$, $\mathcal F(Y)=\mathrm{H}_\bullet(Y,-)$ and 
$\tilde{\mathcal F}(\cA)=\mathrm{HH}(\cA,-)$.  
\item According to \cref{examples-twisted} and \cref{remark-enriched}, $\mathcal F$ is corepresented by $f:Y\mapsto \zeta_Y$ and 
$\tilde{\mathcal F}$ is corepresented by $\tilde{f}:\cA\mapsto \cA^\vee$. 
\end{itemize}

\begin{theorem}\label{thm:functoriality of Sing} 
The square
\[\xymatrix{
\scat{prePoin}^{\mathrm{ft},s}\ar[r]^-{\flat_s}\ar[d]_-\cLoc&\int\mathrm{Tw}^\ell\scat{Mod}_{\mathcal L}\ar[d]\\
\scat{alpreCY}^{\mathrm{ft},s}\ar[r]_-{\flat_s}&\int\mathrm{Tw}^\ell\scat{BiMod}}\]
commutes. Furthermore, for every space $Y$ of finite type, the functor $\sigma_Y:\scat{Mod}_{\mathcal L(Y)} \to \scat{BiMod}_{\mathcal L(Y)} $ commutes with colimits.
\end{theorem}
Here $\scat{prePoin}^{\mathrm{ft},s}\subset\scat{prePoin}^{s}$ is the full subcategory spanned by spaces of finite type, and 
$\scat{alpreCY}^{\mathrm{ft},s}\subset\scat{alpreCY}^{s}$ is the full subcategory spanned by finite type dg-categories. 
\begin{proof}
Let $Y\in \cS^{\mathrm{ft}}$. Thanks to~\cite[proposition 5.3(1)]{BD1}, we get that $\rho^\sigma(Y,k_Y)\simeq\big(Y,\cLoc(Y)\big)$ and 
$\rho^\sigma(Y,\zeta_Y)\simeq\big(Y,\cLoc(Y)^\vee\big)$, meaning that the hypothesis of \cref{lemma:functoriality2} is satisfied. 
Therefore we obtain the commuting square we wanted. 

Recall the functor $\int\mathrm{Tw}^\ell\scat{Mod}_{\mathcal L} \to \int\mathrm{Tw}^\ell\scat{BiMod}$ is induced by the natural transformation 
$\scat{Mod}_{\mathcal L} \Rightarrow \scat{BiMod}_{\mathcal L}, N \mapsto N \otimes_{\mathcal{L}(Y)} M_Y$, where $M_Y$ is the diagonal 
$\mathcal{L} (Y)\otimes \mathcal{L} (Y)\otimes \mathcal {L}^{\op} (Y)$-module defined in \cite[section 5]{BD1}. 
Since the tensor product commutes with colimits, the second statement follows. 
\end{proof}

\begin{remark}\label{remark-cyclic-lift}
It follows from \cite[proposition 5.3]{BD1} that the natural transformation $\mathrm{H}_\bullet(-,k)\Rightarrow \mathrm{HH}\big(\mathcal{L}(-)\big)$ lifts to 
a natural transformation $\mathrm{H}_\bullet(-,k)\Rightarrow \mathrm{HC}^{-}\big(\mathcal{L}(-)\big)$, so that the functor 
$\scat{prePoin}^{\mathrm{ft},s}\to \scat{alpreCY}^{\mathrm{ft},s}$ lifts as $\scat{prePoin}^{\mathrm{ft},s}\to\scat{preCY}^{\mathrm{ft},s}$. 
Since all these functors are above $\mathcal{L}:\cS^{\mathrm{ft}}\to \scat{dgCat}_k^{\mathrm{ft}}$, we also denote them $\mathcal{L}$ (following the abuse 
of notation from the beginning of \cref{ssec-our-examples}). 
\end{remark}


\subsubsection{Moduli of objects}\label{subperf}

We now deal with another example: 
\begin{itemize}
\item $\cC=\scat{dgCat}^\mathrm{ft}$, $\tilde{\cC}=\scat{dSt}_k^\mathrm{\mathrm{perf}\mathbb{L},op}$ and $\rho=\scat{Perf}$ is the moduli of object functor from \cite{ToVa}. 
\item $\mathcal G=\scat{BiMod}$ and $\tilde{\mathcal G}=\scat{QCoh}$.  
\item Following \cref{examples-universal}, $\Delta(\cA)=\cA$, $\tilde\Delta(X)=\mathbb{L}_X$, ${\mathcal F}(\cA)=\mathrm{HH}(\cA,-)$ and $\tilde{\mathcal F}(X)=\Gamma(\mathbb{L}_X \otimes-)$.  
\item According to \cref{examples-twisted} and \cref{remark-enriched}, $\mathcal F$ is corepresented by $f:\cA\mapsto \cA^\vee$ and 
$\tilde{\mathcal F}$ is corepresented by $\tilde{f}:X\mapsto \mathbb{T}_X$. 
\end{itemize}
Let us define a natural transformation $\sigma:\scat{BiMod}\Rightarrow \scat{QCoh}\,\scat{Perf}$.
For any $\cA\in \scat{dgCat}^\mathrm{ft}$ and $M\in \scat{BiMod}_\cA$, consider a cdga $B$ along with a $B$-point $x:\scat{Spec}(B)\to\scat{Perf}_{\cat{A}}$ and the associated $\cA-B$-bimodule $\mathcal X$. 
Then, following~\cite[\S6.2]{BCS}, define $\mathcal{M}\in \cat{Mod}_\cA^\mathrm{perf}$ by 
\[
x^*\mathcal{M}:=\underline{\cat{Map}}_{\scat{Mod}_{\cat{A}^{\op}}}(M{\underset{\cat{A}}{\otimes}}  \mathcal X, \mathcal X)
\simeq \mathbb{R}\underline{\cat{Hom}}_{\scat{Mod}_{\cat{A}^{\op}}}(M\overset{\mathbb{L}}{\underset{\cat{A}}{\otimes}}  \mathcal X, \mathcal X)
\]
and set $\sigma_\cA(M)=\mathcal M^\vee[-1]$ seen in $\scat{QCoh}( \scat{Perf}_\cA)$, so that $\cA$ and $\cA^\vee$ are mapped to $\mathbb L_{\scat{Perf}_{\cat{A}}}$ and $\mathbb T_{\scat{Perf}_{\cat{A}}}[-2]$ respectively.

\begin{lemma}\label{lemma-natural}
The above defines a natural transformation $\sigma:\scat{BiMod}\Rightarrow \scat{QCoh}\,\scat{Perf}$.
\end{lemma}

\begin{proof}
We first have to prove that, for a given dg-category $\cA$, the assignment $M\mapsto \sigma_\cA(M)$ defines a functor $\sigma_A:\scat{BiMod}_{\cA}\to \scat{QCoh}(\scat{Perf}_{\cA})$. 
It is sufficient to prove that for every $B$-point $y: \scat{Spec}(B) \to \scat{Perf}_\cA$ (where $B$ is a cdga) the assignment $M\mapsto y^*\sigma_\cA(M)$ defines a functor $\scat{BiMod}_{\cA}\to \scat{QCoh}(Y)=\scat{Mod}_B$. 
Having models presenting the $\infty$-categories on both sides, it is sufficient to work $1$-categorically on models.  
For every morphism $\varphi:M\to M'$ of $\cA$-bimodules, by precomposition by $\varphi\otimes\mathrm{id}_\mathcal Y$, we have a map
\[
y^*\mathcal M':=\mathbb{R}\underline{\cat{Hom}}_{\scat{Mod}_{\cat{B}^{\op}}}\Big(M'\overset{\mathbb{L}}{\underset{\cat{B}}{\otimes}}  \mathcal Y, \mathcal Y\Big)
\to \mathbb{R}\underline{\cat{Hom}}_{\scat{Mod}_{\cat{B}^{\op}}}\Big(M\overset{\mathbb{L}}{\underset{\cat{B}}{\otimes}}  \mathcal Y, \mathcal Y\Big)=:y^*\mathcal M
\]
and thus a morphism $y^*\mathcal M^\vee[-1]\to y^*\mathcal M'^\vee[-1]$ as wished.

Now we have to prove that the assignment $\cA\mapsto\sigma_\cA$ is natural. Again, we can reason with strict models and on $C$-points (where $C$ is a fixed cdga). 
Consider a dg-functor $g:\cA\to\cB$ in $\scat{dgCat}^\mathrm{ft}$ and a commuting triangle
\[
\xymatrix{
	\scat{Perf}_\cB\ar[rr]^-{h:=g^*}&&\scat{Perf}_\cA\\
	&\mathrm{Spec}(C)\ar[ul]^-y\ar[ur]_-x&}
\]
with associated $\cA-C$-bimodule $\mathcal X$ and $\cB-C$-bimodule  $\mathcal Y$ so that $\mathcal X=g^*\mathcal Y$.
Then, for all $M\in\scat{BiMod}_\cA$ and $N=M \overset{\mathbb{L}}{\underset{\cat{A}^e}{\otimes}} \cB^e$,
\begin{align*}
	y^*\mathcal N:&=\mathbb{R}\underline{\cat{Hom}}_{\scat{Mod}_{\cat{B}^{\op}}}\Big(N\overset{\mathbb{L}}{\underset{\cat{B}}{\otimes}}  \mathcal Y, \mathcal Y\Big)
	=\mathbb{R}\underline{\cat{Hom}}_{\scat{Mod}_{\cat{B}^{\op}}}\Big(\cB \overset{\mathbb{L}}{\underset{\cat{A}}{\otimes}} M \overset{\mathbb{L}}{\underset{\cat{A}}{\otimes}} \cB\overset{\mathbb{L}}{\underset{\cat{B}}{\otimes}}  \mathcal Y, \mathcal Y\Big)\\
	&=\mathbb{R}\underline{\cat{Hom}}_{\scat{Mod}_{\cat{A}^{\op}}}\Big(M\overset{\mathbb{L}}{\underset{\cat{A}}{\otimes}}  g^*\mathcal Y,g^*\mathcal Y\Big)
	=\mathbb{R}\underline{\cat{Hom}}_{\scat{Mod}_{\cat{A}^{\op}}}\Big(M\overset{\mathbb{L}}{\underset{\cat{A}}{\otimes}}  \mathcal X,\mathcal X\Big)
	= :x^*\mathcal M
	=y^*h^*\mathcal M
\end{align*}
Hence $\mathcal N=f^*\mathcal M$, which proves that, on models, that $\sigma_\cB(g_*(M))=h^*(\sigma_\cA(M))$. 
Since $g_*=\cat{BiMod}_g$ and $h^*=(g^*)^*=\scat{QCoh}(\scat{Perf}_g)$, this gives exactly the expected naturality of $\sigma$. 
\end{proof}

\begin{theorem}\label{thm:functoriality of perf}
The square
\[
\xymatrix{
\scat{alpreCY}^{\mathrm{ft},s}\ar[r]^-{\flat_s}\ar[d]_-{\scat{Perf}}&\int\mathrm{Tw}^\ell\scat{BiMod}\ar[d]\\
{k[s]\!\downarrow\!\Gamma(\mathbb{L}_{-}^{\otimes2})[2]}_{|\scat{dSt}_k^{\mathrm{perf}\mathbb{L},\mathrm{op}}} \ar[r]_-{\flat}&\int\mathrm{Tw}^\ell\scat{QCoh}}
\]
commutes. Furthermore, for every finite type dg-category $\cA$, the functor $\sigma_{\cA}:\scat{BiMod}_{\cA} \to \scat{QCoh}_{\scat{Perf}_{\cA}}$ commutes with colimits.
\end{theorem} 
\begin{proof}
Thanks to \cref{lemma-natural}, and the fact that  $\sigma_\cA$ sends $\cA$ and $\cA^\vee$ to $\mathbb L_{\scat{Perf}_\cA}$ and $\mathbb T_{\scat{Perf}_\cA}[-2]$ respectively, we may apply \cref{lemma:functoriality2} to obtain 
the requested commuting square. 
 
The functor $ \int\mathrm{Tw}^\ell\scat{BiMod}\to \int\mathrm{Tw}^\ell\scat{QCoh}$ is defined via the natural transformation
$\sigma: \scat{BiMod}\Rightarrow \scat{QCoh}\,\bPerf \,, \, M \mapsto \cM^\vee[-1]$. Recall that $\cM[-1]$ is defined an a $B$-point $x:\scat{Spec}(B)\to\scat{Perf}_{\cat{A}}$ with associated $\cA-B$-bimodule $\mathcal X$ by 
\[
x^*\mathcal{M}:=\underline{\cat{Map}}_{\scat{Mod}_{\cat{A}^{\op}}}(M{\underset{\cat{A}}{\otimes}}  \mathcal X, \mathcal X)\,.
\]
Hence the assignment $M\mapsto\mathcal M$ sends colimits to limits. Dualizing $(-)^\vee$ sends finite limits to finite colimits and the shift is exact. 
This proves the second claim. 
\end{proof}

\begin{remark}\label{remark-commuting-triangle}
Thanks to~\cite[Proposition 5.3]{BD2} and the naturality of their proofs we have commuting triangle
\[
\xymatrix{
\scat{alpreCY}^{\mathrm{ft},s}=k[s]\!\downarrow\!\mathrm{HH}_{|\scat{dgCat}_k^{\mathrm{ft}}}  \ar[d]_-{\scat{Perf}}\ar[r]^-{\scat{Perf}} 
 & {k[s]\!\downarrow\!\Gamma(\mathbb{L}_{-}^{\otimes2})[2]}_{|\scat{dSt}_k^{\mathrm{perf}\mathbb{L},\mathrm{op}}} \,.\\
\scat{alpreSymp}^{\mathrm{perf}\mathbb{L},2-s,\mathrm{op}} =k[s]\!\downarrow\!\mathcal{A}^2[2]_{|\scat{dSt}_k^{\mathrm{perf}\mathbb{L},\mathrm{op}}} \ar[ur]& 
}
\]
\end{remark}

Putting together \cref{thm:functoriality of Sing}, \cref{remark-cyclic-lift}, \cref{thm:functoriality of perf}, \cref{remark-commuting-triangle} and the commuting diagram \eqref{eq:diag-of-cat} we obtain the following commuting diagram:
  \begin{center}\begin{tikzpicture}
\node(ttl) at (-2.5,3){$\scat{prePoin}^{\mathrm{ft},s}$};
\node(ttr) at (2.5,3){$\int\mathrm{Tw}^\ell\scat{Mod}_{\mathcal L}$};
\node(tl) at (-4,1.5){$\scat{preCY}^{\mathrm{ft},s}$};
\node(bl) at (-5.5,0){$\scat{preSymp}^{\mathrm{perf}\mathbb{L},2-s,\mathrm{op}}$};
\node(tm) at (0,1.5){$\scat{alpreCY}^{\mathrm{ft},s}$};
\node(bml) at (-3,-1.5){$\scat{alpreSymp}^{\mathrm{perf}\mathbb{L},2-s,\mathrm{op}}$};
\node(bmr) at (3,-1.5){${k[s]\!\downarrow\!\Gamma(\mathbb{L}_{-}^{\otimes2})[2]}_{|\scat{dSt}_k^{\mathrm{perf}\mathbb{L},\mathrm{op}}}$};
\node(tr) at (4,1.5){$\int\mathrm{Tw}^\ell\scat{BiMod}$};
\node(br) at (5.5,0){$\int\mathrm{Tw}^\ell\scat{QCoh}$};
\path[->,>=angle 90,font=\small]     (ttl) edge   (tl) ;
\path[->,>=angle 90,font=\small]     (ttl) edge   (tm) ;
\path[->,>=angle 90,font=\small]     (ttl) edge   (ttr) ;
\path[->,>=angle 90,font=\small]     (ttr) edge   (tr) ;
\path[->,>=angle 90,font=\small]     (tm) edge   (bml) ;
\path[->,>=angle 90,font=\small]     (tm) edge   (bmr) ;
\path[->,>=angle 90,font=\small]     (bml) edge   (bmr) ;
\path[->,>=angle 90,font=\small]     (bl) edge   (bml) ;
\path[->,>=angle 90,font=\small]     (bmr) edge   (br) ;
\path[->,>=angle 90,font=\small]     (tl) edge   (tm) ;
\path[->,>=angle 90,font=\small]     (tm) edge   (tr) ;
\path[->,>=angle 90,font=\small]     (tl) edge   (bl) ;
\path[->,>=angle 90,font=\small]     (tr) edge   (br) ;
\end{tikzpicture}\end{center}



\subsection{Non-degenerate diagrams}\label{secmainresult}

We now introduce in very general terms the condition of non-degeneracy and show that the functors of interest 
preserve it. Furthermore, we show that non-degenerate iterated (co)spans form a symmetric monoidal sub-$(\infty, n)$-category of the symmetric monoidal $(\infty,n)$-category of iterated (co)spans constructed 
in \cref{ssec-sym-mon-span}. 

Before defining the non-degeneracy condition we recall two general constructions. We denote by $s$ and $t$ the source and target functors defined on (subcategories of) arrow categories.

\begin{construction}\label{laxarr}
There is a functor $\cat{Arr}({}^*\cat{Cat})\to{}^*\cat{Cat}$, given on object by the assignment $((f,\eta):(\cA,a)\to(\cB,b))\mapsto (\cB,f(a))$.
Indeed, we have a commutative diagram
\[
\xymatrix{
\cat{Arr}({}^*\cat{Cat})\ar[r]^-s\ar[d]&{}^*\cat{Cat}\ar[d]\\
\cat{Arr}(\cat{Cat})\ar[r]^-s&\cat{Cat}}
\]
and a natural transformation $s\Rightarrow t:\cat{Arr}(\cat{Cat})\to\cat{Cat}$ which induces the desired functor 
\[
\cat{Arr}({}^*\cat{Cat})\to\int_{\cat{Arr}(\cat{Cat})} s\to\int_{\cat{Arr}(\cat{Cat})} t\to {}^*\cat{Cat}.
\]
\end{construction}

\begin{construction}\label{construction-kappa}
Consider a functor $\mathcal J:\mathcal P\to\cat{Cat}$, and a diagram $\delta:K\to\int\mathcal J$ associated to some partial order $K$. 
For every $\kappa\in K$ there is a natural functor $K_{\le \kappa}\to\mathcal J(p_\kappa)$, where $\delta(\kappa)=(p_\kappa,m_\kappa\in\mathcal J(p_\kappa))$.
Indeed, thanks to \cref{laxarr} we get a functor 
\[
\cat{Arr}(K)\longrightarrow\cat{Arr}(\int \mathcal J)=\cat{Arr}(\mathcal{P}\times_{\cat{Cat}}{}^*\cat{Cat})\longrightarrow \cat{Arr}({}^*\cat{Cat})\longrightarrow {}^*\cat{Cat}
\]
sending $\sigma\le \kappa$ to $\big(\mathcal J(p_\kappa),\mathcal J(f_{\sigma\le \kappa})(m_\sigma)\in \mathcal J(p_\kappa)\big)$, where $f_{\sigma\le \kappa}:p_\sigma\to p_\kappa$ is induced by $\delta(\sigma\le\kappa)$, 
and making the following diagram commute: 
\[
\xymatrix{
\cat{Arr}(K) \ar[r]\ar[d]^{t} & \cat{Arr}(\int \mathcal J) \ar[r] \ar[d]^{t}& {}^*\cat{Cat} \ar[d] \\
K \ar[r] & \int \mathcal J \ar[r] & \cat{Cat}
}
\]
For every $\kappa$ in $K$ this induces a functor 
\[
K_{\le \kappa}=\cat{Arr}(K)\times_{K}\{\kappa\}\longrightarrow  {}^*\cat{Cat}\times_{\cat{Cat}}\{\mathcal J(p_\kappa)\}=\mathcal J(p_\kappa),
\]
sending $\sigma\le\kappa$ to $\mathcal J(f_{\sigma\le \kappa})(m_\sigma)$.
\end{construction}

The twisted arrow category functor $\cat{Tw}^\ell: \cat{Cat} \to \cat{Cat}$ has a left adjoint $\cat{Tw}^\ell_{!}: \cat{Cat} \to \cat{Cat}$ (see \cite[Appendix A.4]{CHS} for details). 
Given a diagram $\delta: K\to\int\cat{Tw}^\ell\mathcal G$, applying \cref{construction-kappa} with $\mathcal J=\cat{Tw}^\ell\mathcal G$ (and $\mathcal P=\mathcal C$) we obtain a diagram 
$K_{\le \kappa}\to \cat{Tw}^\ell  \mathcal G(c_\kappa)$. By adjunction this yields a diagram $\delta_\kappa: \cat{Tw}^\ell_!K_{\le \kappa}\to\mathcal G(c_\kappa)$. 

\begin{remark}
The category $ \cat{Tw}^\ell_!K$ is described as follows: for every object $\kappa$ we add an object $\kappa^!$ and an arrow  $d_\kappa:\kappa^! \rightarrow \kappa$, and for every arrow $f:\sigma \rightarrow \kappa$ 
we add an arrow $f^!:\kappa^! \rightarrow \sigma^!$ 
such that $f^! d_\sigma f=d_\kappa$. 
Now write $\delta(\kappa)=(c_\kappa, m_\kappa, m_\kappa^!, D_\kappa)$ where $c_\kappa \in \cC$, $m_\kappa, m^!_\kappa \in \cG(c_\kappa)$ and 
$D_\kappa \in \mathrm{Hom}_{\cG(c_\kappa)}(m_\kappa^!,m_\kappa)$. 
Then $\delta_\kappa(\kappa)= m_\kappa$, $\delta_\kappa(\kappa^!)=m^!_\kappa$ and $\delta_\kappa(d_\kappa)=D_\kappa$. 
Furthermore  if $\sigma\le\kappa$ then $\delta(\sigma\le\kappa)$ defines a twisted arrow 
\[
\xymatrixcolsep{5pc}\xymatrix{
\mathcal G(f_{\sigma\le\kappa})(m_\sigma^!) \ar[r]^{\mathcal G(f_{\sigma\le\kappa})(D_\sigma)} & \mathcal G(f_{\sigma\le\kappa})(m_\sigma) \ar[d] \\
\ar[u] m_\kappa^! \ar[r]^{D_\kappa} & m_\kappa
}
\raisebox{-20pt}{=}\,\,\,
\xymatrix{
\delta_\kappa(\sigma^!) \ar[r]^{\delta_\kappa(d_\sigma)} & \delta_\kappa(\sigma) \ar[d]^{\delta_\kappa(\sigma\to\kappa)} \\
\ar[u]^{\delta_\kappa(\kappa^!\to \sigma^!)} \delta_\kappa(\kappa^!) \ar[r]^{\delta_\kappa(d_\kappa)} & \delta_\kappa(\kappa)
}
\]
\end{remark}

\begin{definition}\label{ndegdef}\begin{enumerate}[label=(\roman*)]
\item A diagram $\delta:K\to\int\cat{Tw}^\ell\mathcal G$ is non-degenerate if, for every $\kappa\in K$,\[
\delta_\kappa:\cat{Tw}^\ell_!K_{\le \kappa}\to\mathcal G(c_\kappa)\]
is a colimit diagram.
\item Given a functor $\flat:\cat{H}\to\int\cat{Tw}^\ell\mathcal G$, we say that a diagram $\delta:K\to \cat{H}$ is non-degenerate with respect to $\flat$ if the composition $\flat\delta:K\to\int\cat{Tw}^\ell\mathcal G$ 
is non-degenerate in the above sense.
\end{enumerate}
\end{definition}
In practice, $\flat:\cat{H}\to\int\cat{Tw}^\ell\mathcal G$ will be $\flat_s:k[s]\downarrow\mathcal F\Delta\to\int\cat{Tw}^\ell\mathcal G$. 

\begin{remark}\label{rem-functoriality-of-nd}
Given a functor $\rho:\cat{C}\to\tilde{\cat{C}}$ and a natural transformation $\sigma:\mathcal G\Rightarrow \tilde{\mathcal G}\rho$, we have a lift of $\rho$ to $\rho^\sigma:\int\mathrm{Tw}^\ell\mathcal G\to \int\mathrm{Tw}^\ell\tilde{\mathcal G}$. 
If the functor $\sigma_c:\mathcal G(c)\to \tilde{\mathcal G}(\rho(c))$ preserves finite colimits for every $c\in\cat{C}$, then it sends $\rho^\sigma$ sends non-degenerate diagrams in $\int\cat{Tw}^\ell\mathcal G$ to non-degenerate diagrams in 
$\int\mathrm{Tw}^\ell\tilde{\mathcal G}$. 
\end{remark}

\begin{example}\label{example:one-span}
Consider the diagram $K= (\Lambda^1)^{\op}=\kappa_{00} \rightarrow \kappa_{01} \leftarrow \kappa_{11}$. 
\begin{itemize}
\item First note that $K_{\le\kappa_{00}}=\{\kappa_{00}\}$. Hence $\cat{Tw}^\ell_!K_{\le \kappa_{00}}=(\kappa_{00}^!\to\kappa_{00})\simeq [1]$. Therefore
$\delta_{\kappa_{00}}:\cat{Tw}^\ell_!K_{\le \kappa_{00}}\to \mathcal G(c_{\kappa_{00}})$ is a colimit diagram if and only if $D_{\kappa_{00}}:m_{\kappa_{00}}^!\to m_{\kappa_{00}}$ is an isomorphism. 
\item Similarly, $\delta_{\kappa_{11}}$ is a colimit diagram if and only if $D_{\kappa_{11}}:m_{\kappa_{11}}^!\to m_{\kappa_{11}}$ is an isomorphism. 
\item Finally, $\cat{Tw}^\ell_!K_{\le \kappa_{01}}$ is given by 
	\[ \xymatrix{ & \kappa_{00}^! \ar[r] & \kappa_{00} \ar[rd] & \\
		\kappa_{01}^! \ar[ru] \ar[rd] &&& \kappa_{01}\\
		& \kappa_{11}^! \ar[r] & \kappa_{11} \ar[ru] & }\]
Assuming that the two previous conditions are satisfied, $\delta_{\kappa_{01}}: \cat{Tw}^\ell_!K_{\le \kappa_{01}}\to\mathcal G(c_{\kappa_{01}})$ is a colimit diagram if and only if the following 
	\[ 
	\xymatrix{ & \delta_{\kappa_{01}}(\kappa_{00}^!) \ar[r]^-\sim & \delta_{\kappa_{01}}(\kappa_{00}) \ar[rd] & \\
		\delta_{\kappa_{01}}(\kappa^!_{01}) \ar[ru] \ar[rd] &&& \delta_{\kappa_{01}}({\kappa_{01}})\\
		& \delta_{\kappa_{01}}(\kappa_{11}^!) \ar[r]^-\sim & \delta_{\kappa_{01}}(\kappa_{11})\ar[ru] & }\]
is a pushout square in $\mathcal G(c_{\kappa_{01}})$. 
\end{itemize}
\end{example} 

Recall from \cref{subsec-cospans} that composed $n$-iterated cospans in $\cat{H}$ are cartesian functors $\Sigma^{k_1,\dots,k_n}\to\cat{H}^{\op}$, that-is-to-say right Kan extensions of functors $\Lambda^{k_1,\dots,k_n}\to\cat{H}^{\op}$. 
\begin{definition}
Let $\cat{H}$ be an $\infty$-category having push-outs together with a functor $\flat:\cat{H} \to  \int\cat{Tw}^\ell\mathcal G$. 
We define the subspace $\mathrm{COSPAN}^\mathrm{ndg}_n(\cat{H})_{k_1,\dots,k_n}$ of $\mathrm{COSPAN}_n(\cat{H})_{k_1,\dots,k_n}$ 
spanned by non-degenerate diagrams $(\Lambda^{k_1,\dots,k_n})^{\op}\to\cat{H}$ in the sense of \cref{ndegdef}(2). 
\end{definition} 

Note that in the poset $\Lambda^k$ we have $\Lambda^k_{\geq x}=\Lambda^0$ or $\Lambda^1$ for any $x\in \Lambda^k$, so that $\Lambda^{k_1, \dots, k_n}_{\geq y}=(\Lambda^1)^m$ for some $0\le m\le n$ for any 
$y\in \Lambda^{k_1, \dots, k_n}$. Therefore  $(\Lambda^{k_1, \dots, k_n})^{\mathrm{op}}$ is nondegenerate according to our \cref{ndegdef} if and only if $\Lambda^{k_1, \dots, k_n}$ is according to~\cite[Definition 2.9.2]{CHS}.

Assume that $\mathcal H=k[s]\!\downarrow\!\mathcal F\Delta$ and that we are in the setting of \cref{remark-enriched}, 
and further assume that 
\begin{enumerate}
\item $\mathcal G$ takes values in the reflective sub-$\infty$-category $\scat{stCat}_k\subset \scat{dgCat}_k$ of stable $k$-linear categories; 
\item $\Delta$ preserves finite colimits. 
\end{enumerate}
All three \cref{examples-twisted} fit into this framework. 
\begin{proposition}
The construction $\mathrm{COSPAN}^\mathrm{ndg}_n(\cat{H})$ still defines a functor $\Delta^{n,\op}\to \cat{S}$. 
\end{proposition} 

\begin{proof}
The strategy is similar as the one of the proof of \cite[Theorem 2.9.5]{CHS}. 
Recall that 
\[
\mathrm{COSPAN}_n(\cat{H})_{k_1,\dots,k_n}=\mathrm{Map}(\Lambda^{k_1, \dots, k_n}, \cat{H}^{\op})\,.
\] 
For every morphism $([k_1],\dots,[k_n])\to ([k'_1],\dots,[k'_n])$ in $\Delta^n$ we have a sequence of functors 
\[
\Lambda^{k_1, \dots, k_n}\overset{\iota}{\hookrightarrow}\Sigma^{k_1, \dots, k_n}\overset{j}{\rightarrow}
\Sigma^{k'_1, \dots, k'_n}\overset{\iota'}{\hookleftarrow}\Lambda^{k'_1, \dots, k'_n}
\]
which induces a push-pull map 
\[
(j\circ\iota)^*\iota'_*:
\mathrm{COSPAN}_n(\cat{H})_{k'_1,\dots,k'_n}\rightarrow \mathrm{COSPAN}_n(\cat{H})_{k_1,\dots,k_n}\,,
\]
where the pull is obtained by precomposition and the push by right Kan extension. Our aim is to prove that this map preserves non-degeneracy. 
It is enough to do so for elementary degeneracy maps and elementary face maps in each variable.  For degeneracy maps this follows from ~\cite[Proposition 2.9.9]{CHS}. 
Outer face maps do not raise any difficulty as they are given by inclusions $\Lambda^{k-1}\hookrightarrow\Lambda^k$ which simply consist in adding a copy of 
$\Lambda^1$ on one side. 

Thus we are left dealing with inner face maps; it is enough to deal with 
\[
j=(d^1,\mathrm{id}_{[1]},\dots,\mathrm{id}_{[1]}):([1],[1],\dots,[1])\longrightarrow ([2],[1],\dots,[1])\,.
\]
Let $\delta:(\Lambda^{2,1,\dots,1})^{\op}\to \cat{H}=k[s]\!\downarrow\!\mathcal F\Delta$ be a non-degenerate diagram. 
We have to prove that $\mu=(j\circ\iota)^*\iota'_*\delta:(\Lambda^{1,1,\dots,1})^{\op}\to \cat{H}=k[s]\!\downarrow\!\mathcal F\Delta$ is non-degenerate as well, 
i.e.~that for every $\kappa=(\kappa_1,\dots,\kappa_n)\in K=(\Lambda^{1,1,\dots,1})^{\mathrm{op}}$, 
$\mu_\kappa:\cat{Tw}^\ell_!K_{\leq\kappa}\to \mathcal G(c_\kappa)$ 
is a colimit diagram. Whenever $\kappa_1\in\{(0,0),(1,1)\}$, then $K_{\leq\kappa}=(\Lambda^{2,1\dots,1,\mathrm{op}})_{\leq\kappa}$, and thus 
$\mu_\kappa=\delta_\kappa$ is a colimit diagram. We now assume that $\kappa_1=(0,1)$. If $\kappa_i\in\{(0,0),(1,1)\}$ for some $i\geq2$, then 
$K_{\leq\kappa}=\Lambda^{2,1\dots,1,\mathrm{op}}$, where the number of $1$'s is strictly less than before. Therefore the result follows from 
the case when $\kappa=(0,1)^n$ is the maximal element of $K$: $K_{\leq\kappa}=K$. 
This case is a consequence of \cite[Proposition 2.9.13 \& Corollary 2.9.15]{CHS}\footnote{Observe that we are dealing with cospans while \cite{CHS} deals with spans, 
so that one has to change some limits into colimits in the relevant statements from \cite{CHS}} (as in the proof of \cite[Corollary 2.9.16]{CHS}): 
\begin{itemize}
\item Condition (3) of \cite[Proposition 2.9.13]{CHS} follows from that $\Delta$ preserves finite colimits. 
\item Condition (2) is satisfied by the non-degeneracy assumption. 
\item Condition (1) is satisfied because the square (that is obtained by applying $f_s$ to a square in $\Sigma^{2,1,\dots,1}$) is shifted dual to the another square given by 
applying $\Delta$ to a cocartesian square in $\Sigma^{2,1\dots,1,\mathrm{op}}$. The latter is cocartesian, so that the former is cartesian, 
and thus cocartesian as we assumed that $\mathcal G$ takes values in stable categories. \qedhere
\end{itemize}
\end{proof} 

\begin{proposition}
The functor $\mathrm{COSPAN}^\mathrm{ndg}_n(\cat{H}):\Delta^{n,\op}\to\cat{S}$ is an $n$-uple Segal space and the associated $n$-fold Segal space $\cat{coSpan}^\mathrm{ndg}_n(\cat{H})$ is complete.
\end{proposition} 

\begin{proof}
The first part of the statement is proven in the same way as in \cite[Corollary 2.9.6]{CHS}. 
Indeed, recall that $\Lambda^k\simeq\Lambda^1\amalg_{\Lambda_0}\dots\amalg_{\Lambda_0}\Lambda_1$,so that 
\[
\mathrm{Map}(\Lambda^k, \cat{H}^\mathrm{op})\simeq \mathrm{Map}(\Lambda^1, \cat{H}^\mathrm{op})\times_{\cat{H}^\mathrm{op}}\dots\times_{\cat{H}^\mathrm{op}}\mathrm{Map}(\Lambda^1, \cat{H}^\mathrm{op})\,.
\]
But we have seen that $\Lambda^k\to \cat{H}^\mathrm{op}$ is non-degenerate if and only if $\Lambda^\varepsilon\hookrightarrow\Lambda^k\to   \cC^\mathrm{op}$ is for $\varepsilon\in\{0,1\}$. We thus get 
\[
\mathrm{Map}(\Lambda^k, \cat{H}^\mathrm{op})^\mathrm{ndg}\simeq \mathrm{Map}(\Lambda^1, \cat{H}^\mathrm{op})^\mathrm{ndg}\times_{\cat{H}^{\mathrm{op},\mathrm{ndg}}}\dots\times_{\cat{H}^{\mathrm{op},\mathrm{ndg}}}\mathrm{Map}(\Lambda^1, \cat{H}^\mathrm{op})^\mathrm{ndg}\,,
\]
so that the Segal condition holds for $n=1$. This clearly works independantly in each variable and thus for we get the $n$-fold Segal condition for $n\ge1$. 

The proof of the second part of the statement is parallel to \cite[Proposition 2.9.8]{CHS}: we get it from $\mathrm{coSpan}_n(\cat{H})$ because the co-graph $c\overset{\sim}{\to}c'\overset{=}{\leftarrow}c'$ 
of an equivalence $c\overset{\sim}{\to}c'$ is always non-degenerate. 
\end{proof}

Thanks to \cref{examples-twisted}, we are able to define the following $n$-uple Segal spaces: 
\begin{itemize}
\item $\cat{POIN}_n^s:=\cat{COSPAN}_n^\mathrm{ndg}(\scat{prePoin}^{\mathrm{ft},s})$, 
\item $\cat{CY}_n^s:=\cat{COSPAN}_n^\mathrm{ndg}(\scat{preCY}^{\mathrm{ft},s})$, and 
\item $\cat{LAG}_n^{2-s}:=\cat{COSPAN}_n^\mathrm{ndg}(\scat{preSymp}^{\mathrm{perf}\mathbb{L},2-s,\mathrm{op}})$. 
\end{itemize}
Moreover, using the last diagram of \cref{subsection-examples} together with \cref{rem-functoriality-of-nd}, \cref{thm:functoriality of Sing} and \cref{thm:functoriality of perf}, we obtain morphisms  
\[
\cat{POIN}_n^s\longrightarrow \cat{CY}_n^s\longrightarrow \cat{LAG}_n^{2-s}
\]
of $n$-uple Segal spaces. We then denote by 
\begin{equation}\label{equation-maps-of-n-cats}
\scat{Poin}_n^s\longrightarrow \scat{CY}_n^s\longrightarrow\scat{Lag}_n^{2-s}
\end{equation}
the corresponding sequence of (complete) $n$-fold Segal spaces. 
 
 
\subsection{Non-degeneracy and symmetric monoidal structures}
 
 We can now state the following, which encapsulates Theorems A and C from the introduction, as well as Theorem B when $\cC=k$.
 
 \begin{theorem}\label{thmtfts}
 	The $(\infty,n)$-categories $\scat{Poin}_n^s$, $\scat{CY}_n^s$ and $\scat{Lag}_n^{2-s}$ carry symmetric monoidal structures, as well as the $(\infty,n)$-functors from \eqref{equation-maps-of-n-cats}. 
 \end{theorem}
 
 \begin{proof}
 	For the first part, the line of argument is exactly the same as in~\cite[Proposition 2.10.7]{CHS}.
 	Precisely, in all our cases, the equivalence
 	\[
 	\mathrm{coSpan}_n(k[s]\!\downarrow\! F\Delta)(\varnothing,\varnothing)\simeq 	\mathrm{coSpan}_{n-1}(k[s+1]\!\downarrow\! F\Delta)
 	\]
	of $(\infty,n-1)$-categories restricts to the non-degenerate subcategories. 
 \end{proof}
 
One can prove that every object in any of the symmetric monoidal $(\infty,n)$-category appearing in \cref{thmtfts} is fully dualizable in the sense of \cite[\S2.3]{Lurie-TFT}. 
For $\scat{Poin}_n^s$ and $\scat{Lag}_n^{2-s}$ this is proven in \cite[\S2.11]{CHS}. The proof for $\scat{CY}_n^s$ is completely similar.

\makeatother

\makeatletter

\section{The non-commutative AKSZ functor}

We want to relate the present work to~\cite{CHS}, where a symmetric monoidal $(\infty,n)$-category $\scat{Or}_n^d$ of $d$-oriented derived stacks is constructed.
Given an $s$-Calabi--Yau category $\cC$, $\scat{Perf}_\cC$ is a $(2-s)$-symplectic derived stack, and the associated symmetric monoidal AKSZ$(\infty,n)$-functor 
$\mathrm{aksz}_{\scat{Perf}_\cC}:\scat{Or}_n^d\to\scat{Lag}_n^{2-d-s}$ constructed in~\cite[Theorem C]{CHS} is essentially given by 
$\textsc{Map}(-,\scat{Perf}_\cC)$.

Note the existence of a symmetric monoidal $(\infty,n)$-functor $\scat{Poin}_n^d\to\scat{Or}_n^d$ given by the Betti stack construction $(-)_B$, that identifies 
$\scat{Poin}_n^d$ as a symmetric monoidal sub-$(\infty,n)$-category of $\scat{Or}_n^d$. 

The goal of this section is to prove Theorem D, saying that the induced composition 
\[
\textsc{Map}((-)_B,\scat{Perf}_\cC):\scat{Poin}_n^d\to \scat{Lag}_n^{2-d-s}
\]
is the same as
\[
\scat{Poin}_n^d\to\scat{CY}_n^d\to\scat{CY}_n^{d+s}\to \scat{Lag}_n^{2-d-s}\,,
\]
where the leftmost and rightmost $(\infty,n)$-functors are the ones from \eqref{equation-maps-of-n-cats}, and the middle one is induced by the tensor product with $\cC$, which we construct in the following subsection (hence proving Theorem B along the way).

\subsection{Tensor product of Calabi--Yau categories}

In order to properly define the tensor product of Calabi--Yau categories, we first give an account of the weak K\"unneth theorem for negative cyclic homology.

\begin{lemma}\label{lemma:kuenneth}
	The negative cyclic homology $\cat{HC}^-:\scat{dgCat} \to \scat{Mod}_k$ is lax symmetric monoidal and commutes with the K\"unneth isomorphism of Hochschild homology. That is given $\cA$ and $\cB$  dg categories, then we have a natural map \[
	\cat{HC}^-(\cA) \otimes_k \cat{HC}^-(\cB) \to \cat{HC}^-(\cA \otimes_k \cB)\]
	 such that the following diagram commutes:
	\[ \xymatrix{ \cat{HC}^-(\cA) \otimes_k \cat{HC}^-(\cB) \ar[d] \ar[r]  & \cat{HC}^-(\cA \otimes_k \cB)\ar[d]\\
		\cat{HH}(\cA) \otimes_k \cat{HH}(\cB) \ar[r]  & \cat{HH}(\cA \otimes_k \cB ) } \]
\end{lemma}

\begin{proof}
	Let $\Lambda$  denote Conne's cyclic simplicial category. By \cite{Hoyois} the negative cyclic homology complex can be obtained as follows: we consider the functor 
	\[(-)^\#: \scat{dgCat}_k \to \scat{Mod}_k^{\Lambda^{op} }\]
	 associating to a dg category $\cC$ the cyclic object in simplicial complexes 
	 \[\cC^\#_n := \colim_{ (a_0, \cdots, a_n) \in \mathrm{Obj}(\cC)} \cC(a_0, a_1) \otimes_k \cdots \otimes_k \cC(a_{n-1}, a_n) \otimes_k \cC(a_n, a_0)\]
	  whose geometric realization after restriction to $\Delta$ is the Hochschild homology. In Corollary 2.4 of \cite{Hoyois}, we have an equivalence \[K: \scat{Mod}_k^{\Lambda^{op} } \to  \scat{Mod}_{k[\delta]}\]
	  of  $\infty$-category between the cyclic simplicial objects and the mixed complexes.  The composition of the functors associates to a dg category the mixed Hochschild homology complex. Now we compose with the functor taking $S^1$ homotopy fixed points 
	  \[(-)^{hS^1} :\scat{Mod}_{k[\delta]} \to \scat{Mod}_{k}\]
	   associating to a mixed complex $(C, d, \delta)$ the complex $(C[[u]], d-u \delta)$, where $|u|=2$.  The composition of these three functors yields negative cyclic homology. 
	It is easy to see that $(-)^{\#}$ and $K$ are symmetric monoidal by definition, the functor $(-)^{hS^1} $ on the other hand is only lax monoidal. This finishes the proof and gives a concrete description of the map \[\cat{HC}^-(\cA) \otimes_k \cat{HC}^-(\cB) \to \cat{HC}^-(\cA \otimes_k \cB).\qedhere\] \end{proof}

We can now construct a tensor product over $k$ of higher Calabi--Yau cospans.

\begin{proposition}\label{tenscy}
There is a  natural functor  \[-\otimes_k-: \scat{CY}^{d}_n \times \scat{CY}^{s} _n \rightarrow \scat{CY}^{d+s} _n\] of $(\infty,n)$-categories lifting the usual tensor product of dg categories. 
\end{proposition}

\begin{proof}
	The tensor product \[
	\otimes_k:(\scat{dgCat}_k^\mathrm{ft})^2\to \scat{dgCat}_k^\mathrm{ft}\] 
	 lifts to a natural transformation
	\[\boxtimes_k:\scat{BiMod}^2\Rightarrow \scat{BiMod}(-\otimes_k-)\]
	that in turn yields a functor \[
	\boxtimes_k:\Big(\int\mathrm{Tw}^\ell\scat{BiMod}\Big)^{\!2}\to \int\mathrm{Tw}^\ell\scat{BiMod} \]
	such that\[\xymatrix{
	k[d]\!\downarrow\!\cat{HC}^-_{|\scat{dgCat}_k^\mathrm{ft}} \times k[s]\!\downarrow\!\cat{HC}^-_{|\scat{dgCat}_k^\mathrm{ft}}\ar[r]\ar[d]_-{\otimes_k}&
	(\int\mathrm{Tw}^\ell\scat{BiMod})^2\ar[d]^-{\boxtimes_k}\\
	k[d+s]\!\downarrow\!\cat{HC}^-_{|\scat{dgCat}_k^\mathrm{ft}} 	\ar[r]&
	\int\mathrm{Tw}^\ell\scat{BiMod}
	}\] 
commutes,	where the left downward arrow is given by the previous lemma.
Indeed we may apply~\cref{lemma:functoriality2} because as $\cA\otimes_k\cB$-bimodules we have $(\cA\otimes_k\cB)^\vee\simeq\cA^\vee\boxtimes_k\cB^\vee$. By definition of non-degeneracy and thanks to \cref{rem-functoriality-of-nd}, this implies the result since $\boxtimes_k:\scat{BiMod}^2_{\cA,\cB}\to\scat{BiMod}_{\cA\otimes_k\cB}$ preserves colimits. 
\end{proof}

\subsection{Comparing AKSZ with its noncommutative counterpart}

We can now properly state the main theorem of this section, which implies Theorem D.

\begin{theorem}\label{thm:mapping}
The following diagram of symmetric monoidal $(\infty, n)$-categories 
\[
\xymatrix{ & \scat{Poin}_{n}^d \times \scat{CY}^s_n \ar[dr]^-{ \mathcal{L}\times\mathrm{id}  } \ar[ld]_-{ (-)_B\times\bPerf ~~ } & \\
\scat{Or}^{ d}_n \times  \scat{Lag}_{n} ^{2-s}
 \ar[rd]_-{\textsc{Map}(-,-)~~}& &  \scat{CY}^{d}_n \times \scat{CY}^{s} _n\ar[ld]^-{ ~~\bPerf_{- \otimes_k - } } \\
& \scat{Lag}^{2-d-s}_n & 
}\] 
commutes.
\end{theorem}

We first convince ourselves of the following claim.

\begin{lemma}\label{lemma:objects}
The functors $\textsc{Map}\big((-)_B,\bPerf_-)$ and $\bPerf_{\mathcal L(-)\otimes -}$, going from $\cat{S}^{\mathrm{ft}} \times \scat{dgCat}^{\mathrm{ft}}_k$ to 
$\scat{dSt}_k^{\mathrm{perf}\mathbb{L},\mathrm{op}}$, are naturally equivalent. 
\end{lemma}

\begin{proof}
The $\infty$-category $\scat{dgCat}^{\mathrm{ft}}_k$ is finitely cocomplete and therefore tensored over $\cat{S}^{\mathrm{ft}}$. 
Furthermore, we have that $\cA \otimes M \simeq \cA \otimes_k (k \otimes M)\simeq \cA \otimes_k \cLoc(M)$ naturally in $M \in  \cat{S}^{\mathrm{ft}}$ 
and $\cA \in \scat{dgCat}^{\mathrm{ft}}_k$ (indeed, $\mathcal L$ commutes with colimits, so that $\cLoc(M)\simeq k \otimes M$). 

Similarly, $\scat{dSt}_k^{\mathrm{perf}\mathbb{L}}$ is finitely complete and therefore cotensored over $\cat{S}^{\mathrm{ft}}$. 
Furthermore, we have $[M,X]\simeq \textsc{Map}_{\scat{dSt}_k}(M_B, X)$ naturally in $M \in  \cat{S}^{\mathrm{ft}}$ and 
$X\in \scat{dSt}_k^{\mathrm{perf}\mathbb{L}}$ (indeed, $(-)_B$ commutes with colimits, so that $M_B\simeq \Spec(k)\otimes M$). 

The result follows from the fact that $\bPerf:\scat{dgCat}^{\mathrm{ft}}_k\to \scat{dSt}_k^{\mathrm{perf}\mathbb{L},\mathrm{op}}$ commutes with finite colimits, 
and thus with the tensor structure: 
\[
\textsc{Map}_{\scat{dSt}_k}(M_B, \bPerf_\cA)\simeq [M,\bPerf_\cA] \simeq  \bPerf_{\cA\otimes M}\simeq \bPerf_{\cA \otimes_k \cLoc(M)}
\]
which yields the result as all is natural in $M$ and $\cA$.
\end{proof}

Recall the singular homology functor $\mathrm{H}_\bullet:\cat{S} \to \scat{Mod}_k$ and the singular cohomology functor 
$\mathrm{H}^\bullet: \cat{S} \to \scat{Mod}_k$. These are just the tensor and cotensor structure on $\scat{Mod}_k$, respectively.


\begin{lemma}\label{lem-3.5-ptvv}
	Let $M$ be a topological space and $X$ a derived stack. 
	When $\mathrm{H}^\bullet(M)$ is dualizable, the morphism \[
	\kappa_{\textsc{Map}(M_B,X),M_B}:\scat{DR}(\textsc{Map}(M_B,X)\times M_B)    \to \scat{DR}(\textsc{Map}(M_B,X))\otimes \mathrm{H}^\bullet(M)\]
	constructed in~\cite[\S 2.1]{PTVV} is induced by the tensor and cotensor structures of graded mixed complexes and derived stacks over topological spaces.
\end{lemma}

\begin{proof}
Let $F:\cC \to \mathcal{D}$ between $\infty$-categories allowing small limits and colimits. We have a natural morphism $\colim_D F(-) \to F(\colim_D(-))$ where $D$ is a small diagram. 
In particular, since $\cC$ and $\mathcal{D}$ are then tensored and cotensored over topological spaces, we obtain a natural morphism
\[F(-)\otimes M \to  F(-\otimes M).\]
This morphism coincides with 
\begin{equation}\label{tenscotens}
	F(-)\otimes M \to F([M,-\otimes M])\otimes M \to [M, F(-\otimes M)]\otimes M \to F(-\otimes M),
	\end{equation}
where the first map is the coevaluation, the second map is the natural morphism $F([M,-])\to [M,F(-)]$ and the last map is the evaluation. This follows from the fact that tensoring and cotensoring are adjoints. 

We apply this to the choice of $F=\scat{DR}$, $\cC=\scat{dSt^{op}}$ and $\mathcal{D}=\scat{Mod}^{\epsilon-\mathrm{gr}}_k$ the category of graded mixed complexes.
In derived stacks, the tensor and cotensor structure over topological spaces is given by $-\otimes M=-\times M_B$ and $ [M,-]=\textsc{Map}(M_B,-)$. 
In the category of graded mixed complexes the tensor and cotensor structure  over topological spaces is given by $ -\otimes M = -\otimes \mathrm{H}_\bullet(M)$ and $ [M,-]=-\otimes \mathrm{H}^{\bullet}(M)$.
Now the composition of the two rightmost maps in \eqref{tenscotens}, tensored by $\mathrm{H}^\bullet(M)$, yields $\kappa_{\textsc{Map}(M_B,X),M_B}$.
\end{proof}

\begin{lemma} \label{lemma:symplectic}
Let $M$ be a $d$-pre-Poincaré space and $\cA$ a $s$-pre-Calabi--Yau dg category. Then the isomorphism $\textsc{Map}(M_B, \bPerf_{\cA})\simeq \bPerf_{\mathcal{L}(M) \otimes \cA}$ identifies the canonical $(2-d-s)$-pre-symplectic structure on $\textsc{Map}(M_B, \bPerf_{\cA})$ given by \cite{PTVV} with the one on $\bPerf_{\mathcal{L}(M) \otimes \cA}$ given in \cite{BD1}. Furthermore, the identification of forms is functorial both in $M$ and $\cA$. 
\end{lemma}

\begin{proof}
We have to show that the following diagram of complexes commute. 
The right column corresponds to the presymplectic structure constructed in \cite{BD1} while the left one corresponds to the presymplectic structure given in \cite{PTVV} 
thanks to \cref{lem-3.5-ptvv}).
\begin{equation}\label{maindiagbis}
	\xymatrix{ 
		\cA^{p,\mathrm{cl}}[p](\scat{Perf}_{\cA})  \otimes \mathrm{H}_\bullet(M) \ar[d]^-{\mathrm{ev}^*} &\cat{HC}^-(\cA) \otimes \mathrm{H}_\bullet(M)  \ar[d]	\ar[l] \\
		\cA^{p,\mathrm{cl}}[p](\textsc{Map}( M_B, \bPerf_{\cA}) \times M_B ) \otimes \mathrm{H}_\bullet(M)   \ar[d]   &  \cat{HC}^-(\cA) \otimes \cat{HC}^-(\cLoc(M))  \ar[d] \\
		\cA^{p,\mathrm{cl}}[p](\textsc{Map}(M_B, \bPerf_{\cA}) ) \otimes \mathrm{H}^\bullet(M)   \otimes \mathrm{H}_{\bullet}(M) \ar[d]& \cat{HC}^-(\cA \otimes \cLoc(M)) \ar[d]\\
	 \cA^{p,\mathrm{cl}}[p] (\textsc{Map}(M_B, \bPerf_{\cA} ) )  & \ar[l]^-{\textrm{\cref{lemma:objects}}}_{\simeq}\cA^{p,\mathrm{cl}}[p] (\bPerf_{\cA\otimes\mathcal L(M)}).
	}
\end{equation}
The left column is obtained by applying the weight $\geq2$ realization functor to the map \eqref{tenscotens} (that we got by adjunction of the tensor and cotensor structure of graded mixed complexes and derived stacks over spaces), and so is just the natural morphism 
\[
\cA^{p,\mathrm{cl}}[p](\scat{Perf}_{\cA})\otimes M\to \cA^{p,\mathrm{cl}}[p] ([M,\bPerf_{\cA}] )\,.
\]
Similarly, the composition 
\[
\cat{HC}^-(\cA) \otimes_k \mathrm{H}_\bullet(M) \to \cat{HC}^-(\cA) \otimes_k \cat{HC}^-(\cLoc(M)) \to \cat{HC}^-(\cA \otimes_k \cLoc(M))
\] 
coincides with the natural morphism $\cat{HC}^-(\cA)\otimes M\to \cat{HC}^-(\cA\otimes M)$. 
The commutativity now follows from the naturality of $ \cat{HC}^-\Rightarrow \cA^{p,\mathrm{cl}}[p]\circ \scat{Perf}$.  
  \end{proof}

\begin{proof}[Proof of~\cref{thm:mapping}]
The \cref{lemma:objects} yields the following triangle, where $\mathrm{H}_\bullet$, $\cat{HC}^-$ and $\mathcal{A}^{2,\mathrm{cl}}$ are restricted to $\cat{S}_k^\mathrm{ft}$, $\scat{dgCat}_k^\mathrm{ft}$ and $\scat{dSt}_k^{\mathrm{perf}\mathbb{L},\mathrm{op}}$ respectively, 
\[
\xymatrix{ & k[d]\!\downarrow\!\mathrm{H}_\bullet \times k[s]\!\downarrow\!\cat{HC}^- \ar[dr]^-{ \mathcal{L}\times \mathrm{id}  } \ar[dd]_-{  \textsc{Map}((-)_B\times \bPerf)} & \\
& & k[d]\!\downarrow\!\cat{HC}^- \times k[s]\!\downarrow\! \cat{HC}^-  \ar[ld]^-{~~~\bPerf_{- \otimes_k - } } \\
	& k[d+s]\!\downarrow\!\mathcal{A}^{2,\mathrm{cl}}[2]& 
}\]
which is commutative thanks to~\cref{lemma:symplectic}. 
Let us now consider the following diagram
\[
\begin{tikzpicture}[scale=.78]
	\node  (TopRight) at (3,2) {$k[d+s]\!\downarrow\!\mathcal{A}^{2,\mathrm{cl}}[2]$};
	\node  (BotRight) at (3,-3) {$\int  \mathrm{Tw}^\ell \scat{QCoh}$};
	\node  (BackTopLeft) at (-2,4) {$k[d]\!\downarrow\!\mathrm{H}_\bullet \times k[s]\!\downarrow\!\cat{HC}^-$};
	\node  (BackTopRight) at (6,4) {$k[d]\!\downarrow\!\cat{HC}^- \times k[s]\!\downarrow\! \cat{HC}^- $};
	\node   (BackBotLeft) at (-2,-1) {$\int  \mathrm{Tw}^\ell  (\scat{Mod}_\mathcal{L}\times\scat{BiMod})$}; 
	\node  (BackBotRight) at (6,-1) {$\int  \mathrm{Tw}^\ell \scat{BiMod}^2$};
	\path[->,>=angle 90,font=\small]  
	(TopRight) edge (BotRight) ;
	\path[->,>=angle 90,font=\small]  
	(BackTopLeft) edge (BackTopRight) ;
	\path[->,>=angle 90,dashed,font=\small]  
	(BackBotLeft) edge (BackBotRight) ;
	\path[->,>=angle 90,font=\small]  
	(BackTopLeft) edge (BackBotLeft) ;
	\path[->,>=angle 90,font=\small]  
	(BackTopRight) edge (BackBotRight) ;
	\path[<-,>=angle 90,font=\small]  
	(TopRight) edge (BackTopLeft) ;
	\path[<-,>=angle 90,font=\small]  
	(BotRight) edge (BackBotLeft) ;
	\path[<-,>=angle 90,font=\small]  
	(TopRight) edge (BackTopRight) ;
	\path[<-,>=angle 90,font=\small]  
	(BotRight) edge (BackBotRight) ;
\end{tikzpicture}
\]

\noindent where the back face commutes thanks to \cref{thm:functoriality of Sing} and the right one thanks to \cref{thm:functoriality of perf} and \cref{tenscy}. 
We now define the left face. For any $(Y,\cA)\in \cat{S}_k^\mathrm{ft}\times\scat{dgCat}_k^\mathrm{ft}$, let us consider the evaluation \[Y_B\times\textsc{Map}(Y_B,\scat{Perf}_\cA)\overset{\mathrm{ev}}{\longrightarrow}\scat{Perf}_\cA\] along with the projections $p_1,p_2$. We have a natural transformation \[
\sigma:\scat{Mod}_\mathcal{L}\times\scat{BiMod} \Rightarrow \scat{QCoh} \big(\textsc{Map}((-)_B, \scat{Perf}_{-})\big)
\]
given by \[
(N,M)\in  \scat{Mod}_{\mathcal{L}(Y)}\times\scat{BiMod}_\cA \mapsto p_{2*}(p_1^*N\otimes \mathrm{ev}^*\mathcal M^\vee[-1])
\]
where we use the equivalence between $\scat{Mod}_{\mathcal{L}(Y)}$ and $\scat{QCoh}(Y_B)$. Note that for every $Y$ and $\cA$ of finite type, $\sigma_{Y,\cA}$ commutes with finite colimits. 
This defines the left face. We may then apply~\cref{lemma:functoriality2} thanks to~\cite[Corollary B.10.23]{CHS} which states that 
\[
\mathbb T_{\textsc{Map}(Y_B,\scat{Perf}_\cA)}=p_{2*}\mathrm{ev}^*\mathbb T_{\scat{Perf}_\cA}
\]
to get that this left face commutes.
Note that the bottom triangle do not commute a priori. However, to get non-degeneracy, we only need the triangle image of the downward vertical maps to commute, which is a consequence of the commutativity of the top and side faces. Finally, our functors preserve finite colimits as was shown in \cref{thm:functoriality of Sing} and \cref{thm:functoriality of perf}, so we can conclude using~\ref{rem-functoriality-of-nd}.
\end{proof}

\makeatother

\appendix

\makeatletter


\section{Background material on $\infty$-categories}\label{appendixA}


\subsection{Grothendieck construction}

Let $\cC$ be a small $\infty$-category. According to~\cite[Theorem 3.2.0.1]{L} the unstraighening functor, a-k-a ($\infty$-categorical) Grothendieck construction, induces 
an equivalence 
\[
\int:\cat{Fun}(\cC^\mathrm{op},\cat{Cat})\tilde{\longrightarrow}\cat{CarFib}(\cC),
\]
where $\cat{CarFib}(\cC)$ is the $\infty$-category of cartersian fibrations over 
$\cC$ with small fibers (morphisms are cartesian functors). Objects of $\int F$ are pairs $\big(c\in\cC,\cat{M}\in F(c)\big)$, and morphisms $(c,\cat{M})\to (d,\cat{N})$ 
are pairs $(f,\varphi)$ where $f:c\to d$ is a morphism in $\cC$ and $\varphi:\cat{M}\to F(f)(\cat{N})$ is a morphism in $F(d)$. 

Similarly it gives an an equivalence 
\[
\int:\cat{Fun}(\cC,\cat{Cat})\tilde{\longrightarrow}\cat{coCarFib}(\cC),
\]
where $\cat{coCarFib}(\cC)$ is the $\infty$-category of cocartersian 
fibrations over $\cC$ with small fibers. Objects of $\int F$ are pairs $\big(c\in\cC,\cat{M}\in F(c)\big)$, and morphisms $(c,\cat{M})\to (d,\cat{N})$ 
are pairs $(f,\varphi)$ where $f:c\to d$ is a morphism in $\cC$ and $\varphi:F(f)(\cat{M})\to \cat{N}$ is a morphism in $F(d)$. 

These equivalences restrict to the following for $\infty$-groupoids-valued  functors: 
\[
\begin{tikzcd}        
\cat{Fun}(\cC^\mathrm{op},\cat{Cat})\arrow{r}{\sim}&\cat{CarFib}(\cC)\\
\cat{Fun}(\cC^\mathrm{op},\cS)\arrow{r}{\sim}\arrow[hook]{u}&\cat{RFib}(\cC)\arrow[hook]{u}\end{tikzcd}
\qquad\textrm{and}\qquad
\begin{tikzcd}        
\cat{Fun}(\cC,\cat{Cat})\arrow{r}{\sim}&\cat{coCarFib}(\cC)\\
\cat{Fun}(\cC,\cS)\arrow{r}{\sim}\arrow[hook]{u}&\cat{LFib}(\cC)\arrow[hook]{u}\end{tikzcd}
\]
where $\cat{RFib}(\cC)$ (resp.~$\cat{LFib}(\cC)$) is the $\infty$-category of right fibrations (resp.~left fibrations) over $\cC$. 

According to~\cite[Corollary A.32]{GHN} (see also~\cite[Remark 1.13]{MG}), all these equivalences are natural in $\mathcal C$, meaning that they lead to 
the following diagrams of cartesian fibrations over $\cat{Cat}$: 
\[
\begin{tikzcd}
\int\cat{Fun}(-^\mathrm{op},\cat{Cat})\arrow{r}{\sim}&\int\cat{CarFib}\\
\int\cat{Fun}(-^\mathrm{op},\cS)\arrow{r}{\sim}\arrow[hook]{u}&\int\cat{RFib}\arrow[hook]{u}
\end{tikzcd}
\qquad\textrm{and}\qquad
\begin{tikzcd}
\cat{Cat}_{\!/\!\!/\cat{Cat}}:=\int\cat{Fun}(-,\cat{Cat})\arrow{r}{\sim}&\int\cat{coCarFib}(\cC)\\
\cat{Cat}_{\!/\!\!/\mathcal S}:=\int\cat{Fun}(-,\cS)\arrow{r}{\sim}\arrow[hook]{u}&\int\cat{LFib}\arrow[hook]{u}
\end{tikzcd}
\]
In other words, every natural transformation $H\Rightarrow G\circ F$ gives rise to a functor $\int H\to G^*\int F$. 
In particular $\int G\circ F\simeq G^*\int F$. 


\subsection{(Co)slice categories}

Note that applying the Grothendieck construction\footnote{There are size issues, that we will ignore in this paper (they can be resolved using three universes $\mathbf U\in\mathbf V\in\mathbf W$). } to 
$\mathrm{id}:\cat{Cat}\to\cat{Cat}$ allows to define the lax coslice category ${}^*\cat{Cat}:=\int\mathrm{id}$: objects of $^*\cat{Cat}$ are pairs $(\cA,a)$ where $a$ 
is an object of a category $\cA$, and morphisms $(\cA,a)\to(\cB,b)$ are pairs $(f,\eta)$ where $f:\cA\to\cB$ is a functor and $\eta:f(a)\to b$ is a morphism in $\cB$. 
It is universal in the following sense, thanks to the naturality of the Grothendieck construction: for every functor $G$, the natural equivalence $G\simeq G\circ \mathrm{id}$ gives the equivalence $\int G\simeq \int G\circ \mathrm{id}\simeq G^*\int \mathrm{id}=G^*({}^*\cat{Cat})$. 

Similarly, applying the Grothendieck construction to $\cat{S}\hookrightarrow\cat{Cat}$ gives the coslice category ${}^*\cat{S}$:  objects of $^*\cat{S}$ are pairs 
$(X,x)$ where $x$ is an element of a space $X$, and morphisms $(X,x)\to(Y,y)$ are pairs $(f,\eta)$ where $f:X\to Y$ is a map and $\eta:f(x)\to y$ is a path $Y$. 

\begin{example}[coslice category]\label{lemmgrothcosl}
Let $c$ be an object in a small category $\cC$. Then $\int\cat{Map}_{\cC}(c,-)\simeq {}^{c/}\cC$. Indeed, the square 
\[
\xymatrix{
{}^{c/}\cC\ar[rr]\ar[d]&&{}^*\cat{S}\ar[d]\\
\cC\ar[rr]^-{\cat{Map}_{\cC}(c,-)}&&\cat{S}.
}
\]
is cartesian. 
For instance, if $\cC=\scat{Mod}_k$ and $c=k$ then $\int|-|\simeq {}^{k/}\scat{Mod}_k$. Therefore, for a functor $F:\cC\to\scat{Mod}_k$, 
\[
\int|F|\simeq \cC\times_{\scat{Mod}_k}{}^{k/}\scat{Mod}_k=k\!\downarrow\!F.
\]
\end{example}

\begin{example}[lax slice category]
Let $\cat{D}$ be a category. Then one has the lax slice category 
\[
\cat{Cat}_{\!/\!\!/\cat{D}}:=\int\cat{Fun}(-,\cat{D})\,.
\]
In fact, for any $(\infty,2)$-category $\mathbf{C}$, and any object $\cat{D}$ of $\mathbf{D}$, the lax slice category $\mathbf{C}_{\!/\!\!/\cat{D}}$ is the 
Grothendieck construction of the functor $\mathbf{D}(-,\cat{D}):\tau_1\mathbf{C}^{\mathrm{op}}\to \cat{Cat}$, where $\tau_1\mathbf{C}$ is the maximal $(\infty,1)$-subcategory of $\mathbf{C}$. 
\end{example}

\begin{lemma}\label{lemmpullpull}
Consider the pullbacks\[
\xymatrix{\mathcal X\ar[r]\ar[d]&{}^*\cat{Cat}\ar[d]\\
\cat{Cat}_{\!/\!\!/\mathcal S}\ar[r]^-s&\cat{Cat}} 
\raisebox{-.6cm}{\text{ and }}
\xymatrix{\mathcal Y\ar[r]\ar[d]&{}^*\cat{Cat}\ar[d]\\
\cat{Cat}_{\!/\!\!/\mathcal S}\ar[r]^-\int&\cat{Cat}.} 
\]
Then we have $\mathcal Y\simeq\mathcal X\underset{\mathcal S}{\times}{}^*\mathcal S$, where the functor $\mathcal X\rightarrow {\mathcal S}$ is induced by $s\Rightarrow t:\cat{Cat}_{\!/\!\!/\mathcal S}\to\cat{Cat}$. 
\end{lemma}

\begin{proof}
There is a natural transformation $\int\Rightarrow s:\cat{Cat}_{\!/\!\!/\mathcal S}\to\cat{Cat}$ that induces $\mathcal Y\to\mathcal X$.
By naturality of the Grothendieck construction, we have for every $F:\cC\to\cS$ a functor\[
\int F=\int\mathrm{id}_\cS\circ F\to\int\mathrm{id}_\cS={}^*\cS\]
which gives a natural transformation $\int\Rightarrow\mathrm{const}({}^*\cS):\cat{Cat}_{\!/\!\!/\mathcal S}\to\cat{Cat}$. It thus implies in turn a functor \[
\mathcal Y\to\int(\star\overset{{}^*\cS}{\longrightarrow}\cat{Cat})={}^*\cS.\]
The diagram\[\xymatrix{
\mathcal Y\ar[r]\ar[d]&{}^*\mathcal S\ar[d]\\
\mathcal X\ar[r]&\mathcal S}\]
is now commutative and yields a functor $\mathcal Y\to \mathcal X\underset{\mathcal S}{\times}{}^*\mathcal S$ which is actually a functor of fibrations over $\cat{Cat}_{\!/\!\!/\mathcal S}$. Let us see why their fibers are equivalent. Consider a pullback
\[\square=\raisebox{.6cm}{\xymatrix{{}^*\cC\ar[r]\ar[d]&{}^*\mathcal S\ar[d]\\\cC\ar[r]^-F&\mathcal S}},\] 
yielding\[
\xymatrix{\cC\ar[r]\ar[d]&\mathcal X\ar[r]\ar[d]&{}^*\cat{Cat}\ar[d]\\
{\star}\ar[r]^-F&\cat{Cat}_{\!/\!\!/\mathcal S}\ar[r]^-{s}&\cat{Cat}} 
\raisebox{-.6cm}{\text{ and }}
\xymatrix{{}^*\cC\ar[r]\ar[d]&\mathcal Y\ar[r]\ar[d]&{}^*\cat{Cat}\ar[d]\\
{\star}\ar[r]^-F&\cat{Cat}_{\!/\!\!/\mathcal S}\ar[r]^-{\int}&\cat{Cat}.} 
\]
We see that the fibers of $\mathcal X\underset{\mathcal S}{\times}{}^*\mathcal S$ and $\mathcal Y$ over $F$ are both equivalent to ${}^*\cC=\cC\underset{\mathcal S}{\times}{}^*\mathcal S$, which concludes the proof.
\end{proof}


\section{The case of right Calabi--Yau structures}\label{appendixC}

Recall (from \cref{ssec-our-examples} and \cref{examples-twisted}) the sequence of functors 
\[
\scat{preCY}^{\mathrm{sm},s}=k[s]\!\downarrow\!\mathrm{HC}^-_{|\scat{dgCat}_k^{\mathrm{sm}}}
\longrightarrow
\scat{alpreCY}^{\mathrm{sm},s}=k[s]\!\downarrow\!\mathrm{HH}_{|\scat{dgCat}_k^{\mathrm{sm}}}
\longrightarrow\int\mathrm{Tw}^\ell\scat{BiMod}
\]
We introduce the $n$-uple Segal space $\ell\cat{CY}_n^s:=\cat{COSPAN}_n^\mathrm{ndg}(\scat{preCY}^{\mathrm{sm},s})$, as well as its 
associated (complete) $n$-fold Segal space $\ell\scat{CY}_n^s:=\cat{coSpan}^\mathrm{ndg}_n(\scat{preCY}^{\mathrm{sm},s})$. 
As in \cref{thmtfts} (see also \cite[Proposition 2.10.7]{CHS}) it carries a symmetric monoidal structure. 

\medskip

Our goal in this appendix is to explain how the above works as well for right Calabi--Yau structures in the sense of \cite[\S3.1]{BD1}. 
We first consider the full subcategory $\scat{dgCat}_k^{\mathrm{pr}}\subset \scat{dgCat}_k$ of proper dg-categories: a dg-category $\cA$ is proper 
if, for every $a,b\in \cA$, $\cA(a,b)$ is perfect in $\scat{Mod}_k$. We then consider the $k$-linear dual of the cyclic homology functor 
(mind the op's) 
\[
\cat{HC}^*:\scat{dgCat}_k^{\mathrm{pr},\mathrm{op}}\to\scat{Mod}_k
\]
and introduce the categories 
\[
r\scat{preCY}^{\mathrm{pr},s,\mathrm{op}}=k[s]\!\downarrow\!\mathrm{HC}^*_{|\scat{dgCat}_k^{\mathrm{pr},\mathrm{op}}}
\longrightarrow
r\scat{alpreCY}^{\mathrm{pr},s,\mathrm{op}}=k[s]\!\downarrow\!\mathrm{HH}^*_{|\scat{dgCat}_k^{\mathrm{pr},\mathrm{op}}}\,.
\]
Here $(-)^*=\underline{\cat{Map}}_{\scat{Mod}_k}(-,k)$. 
We then apply the general framework of \cref{ssec-abstract-nonsense} with 
\begin{itemize}
\item $\cat{C}=\scat{dgCat}_k^{\mathrm{pr},\mathrm{op}}$ and $\mathcal G=\scat{BiMod}^{\mathrm{op}}$ with 
\[
(f:\cA\to\cB)\longmapsto (f^*:\scat{BiMod}_{\cB}^{\mathrm{op}}\to \scat{BiMod}_{\cA}^{\mathrm{op}})\,.
\]
\item $\Delta:\cA\mapsto\cA$ and $\mathcal F:\cA\mapsto \cat{HH}(\cA,-)^*$. 
\end{itemize}
On objects we get 
\begin{eqnarray*}
\flat_s\,:\,r\scat{alpreCY}^{\mathrm{pr},s,\mathrm{op}} & \longrightarrow & \int\,\mathcal U\scat{BiMod}^{\mathrm{op}}\\
\big(\cA,c:\cat{HH}(\cA)\to k[-s] \big) & \longmapsto & \big(\cA,|\cat{HH}(\cA-,)^*[-s]|,\cA,c\big)
\end{eqnarray*}
Whenever $\cA$ is proper $\cat{HH}(\cA,-)^*$ is corepresented (in $\scat{BiMod}^{\mathrm{op}}_\cA$) by $\cA^*$, defined by $\cA^*(a,b)=\cA(b,a)^*$, 
so that the above factors through 
\begin{eqnarray*}
\flat_s\,:\,r\scat{alpreCY}^{\mathrm{pr},s,\mathrm{op}} & \longrightarrow & \int\,\mathrm{Tw}^\ell\scat{BiMod}^{\mathrm{op}}\\
\big(\cA,c:\cat{HH}(\cA)\to k[-s] \big) & \longmapsto & \big(\cA,f_s(\cA),\cA,c\big)
\end{eqnarray*}
where $f_s(\cA)=\cA^*[-s]$ and $c:\cat{HH}(\cA)\to k[-s]$ leads to 
$c^\flat:\cA^*[-s]\to \cA$ in $\scat{BiMod}_{\cA}^{\mathrm{op}}$ (that is a morphism $\cA\to\cA^*[-s]$  of $\cA$-bimodules). 

We therefore can define the $n$-uple Segal space $r\cat{CY}_n^{s}:=\cat{COSPAN}_n^\mathrm{ndg}(r\scat{preCY}^{\mathrm{pr},s,\mathrm{op}})$, as well as its 
associated (complete) $n$-fold Segal space $\ell\scat{CY}_n^{s}:=\cat{coSpan}^\mathrm{ndg}_n(\scat{preCY}^{\mathrm{pr},s,\mathrm{op}})$. 
Objects are indeed pairs $\big(\cA,c:\cat{HC}(\cA)\to k[-s] \big)$ such that the induced morphism of bimodules $\cA[s]\to \cA^*$ is an equivalence 
(as in \cite[Definition 3.2]{BD1}). 

We strongly believe that the construction from \cite[Theorem 3.1]{BD1} can be promoted to a symmetric monoidal $(\infty,n)$-functor 
\[
\ell\scat{preCY}^{\mathrm{pr},s} \longrightarrow r\scat{preCY}^{\mathrm{pr},-s}
\]
sending iterated left Calabi--Yau cospans to iterated right Calabi--Yau spans. 

\begin{remark}
Note that for almost Calabi--Yau structures, cyclic homology is not involved and one can just ask for a bimodule equivalence $\cA[s]\to\cA^*$. 
Hence one can simply replace the functor $\cat{HH}(\cA,-)^*$ by the one that is corepresented (in $\scat{BiMod}^{\mathrm{op}}_\cA$) by $\cA^*$. 
We let the reader check that in this case we obtain an $(\infty,n)$-category such that: 
\begin{itemize}
\item Objects are the weak right Calabi--Yau categories in the sense of \cite[Definition 2.4]{gorsk} (without the degreewise finiteness assumption);
\item Morphisms (resp.~$2$-morphisms) coincide with the weak right Calabi--Yau spans (resp.~right Calabi--Yau $2$-spans) of \cite{gorsk} 
(again without the degreewise finiteness assumption).
\end{itemize}
\end{remark}

\makeatother 


\section{Proof of an $(\infty,2)$-categorical lemma}\label{appendix}

This section is devoted to the proof of~\cite[Lemma 5.33]{BCS}. 
See~\cite[Theorem 2.2.23]{Grataloup} for an analogous result in the shifted symplectic/lagrangian context. 
We begin with the following lemma, analogous to~\cite[Proposition 2.17]{ABB}

Consider a $4$-tuple $(\cat{C},\mathcal G,\Delta,\mathcal F)$ as in Section 2, and set $\mathbf C_m^n=\mathrm{coSpan}_m^\mathrm{ndg}(k[n]\!\downarrow\!\mathcal F\Delta)$.

\begin{lemma}\label{ndequiv}
	Consider two cospans $\phi:\Upsilon\to\cA$ and $\psi:\Upsilon\to\cB$ in $\mathbf C_1^n(\Upsilon,\varnothing)$. Denote by $\mathcal P=\cA\amalg_{\Upsilon}\cB$ their composition in $\mathbf C_1^n(\varnothing,\varnothing)=\mathbf C_0^{n+1}$.
	Denote by $c:k[n]\to\mathcal F\Delta(\Upsilon)$ the non-degenerate structure on $\Upsilon$, and by $\eta_\cA$ and $\eta_\cB$ the homotopies $\phi(c)\sim0$ and $\psi(c)\sim0$. Consider a morphism $
	\nabla:\mathcal P\rightarrow \mathcal X$
	in $\cC$ such that $\alpha:=\nabla\circ\iota_\cA$ and $\beta:=\nabla\circ\iota_\cB$ are equivalences, where $\iota_A$ and $\iota_B$ denote the natural morphisms $\cA\to\mathcal P$ and $\cB\to\mathcal P$ in $\cC$. If we are given an homotopy $\xi$ between $\alpha(\eta_\cA)$ and $\beta(\eta_\cB)$, then $\nabla\in\mathbf C_1^{n+1}(\mathcal P,\varnothing)$.
\end{lemma}

\begin{proof}
	Let us first summarize the situation with the following diagram:
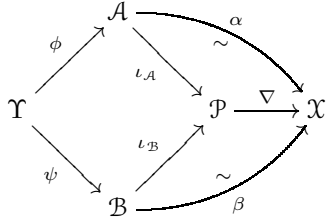
\[\xymatrix{
		&\cA\ar_-{\iota_\cA}[rd]\ar^-\alpha_-[@!-15]{\sim}@/{}^{1pc}/[rrd]&&\\
		\Upsilon\ar^-\phi[ru]\ar_-\psi[rd]&&\mathcal P\ar^\nabla[r]&\mathcal X\\
		&\cB\ar^-{\iota_\cB}[ru]\ar_-\beta^-[@!25]{\sim}@/{}_{1pc}/[rru]&&
	}\]
	Thanks to $\xi$, we have a homotopy commuting diagram\[\xymatrix{
		k[n+1]\ar[r]\ar[d]&\mathcal F\Delta(\mathcal P)\ar[d]\\
		0\ar[r]&\mathcal F\Delta(\mathcal X)}\]
	so that $\nabla\in \mathrm{coSpan}_1(k[n+1]\!\downarrow\!\mathcal F\Delta)(\mathcal P,\varnothing)$.
	
	Let us now prove the non-degeneracy. 
	The composition (recall~\ref{remark-enriched} for the notation $f_\bullet$) \[
	f_n\mathcal X\to\mathcal G(\nabla)(f_n\mathcal P)\to \mathcal G(\alpha)(f_n\cA)\]
	induces 
	\begin{equation}
		\label{keyeq}
		\Big[\mathcal G(\nabla)(f_n\mathcal P)\to \mathcal G(\alpha)(f_n\cA)\Big]\simeq\Big[f_n\mathcal X\to\mathcal G(\nabla)(f_n\mathcal P)\Big][1].
	\end{equation}
	as $\alpha$ is an equivalence.
	The left-hand side is isomorphic to $[\mathcal G(\beta)(f_n\mathcal B)\to \mathcal G(\alpha\circ \phi)(f_n\Upsilon)]$
	by definition of $\mathcal P$, and since $\psi\in\mathbf C_1^n(\Upsilon,\varnothing)$, we have a (co)cartesian square\[\xymatrix{
		f_n\cB\ar[r]\ar[d]&\mathcal G(\psi)(f_n\Upsilon\simeq\Delta\Upsilon)\ar[d]\\
		0\ar[r]&\Delta\cB}\]
	and then \[\xymatrix{
		\mathcal G(\beta)(	f_n\cB))\ar[r]\ar[d]&\mathcal G(\beta\circ\psi)(f_n\Upsilon\simeq\Delta\Upsilon)\ar[d]\\
		0\ar[r]&\mathcal G(\beta)(\Delta\cB)}\]
	where we recognise the left-hand side of~(\ref{keyeq}) on the top row, hence we get a (co)cartesian square\[\xymatrix{
		f_n\mathcal X[1]\ar[r]\ar[d]&\mathcal G(\nabla)(f_n\mathcal P)[1]\ar[d]\\
		0\ar[r]&\mathcal G(\beta)(\Delta\cB)}\]
	as expected, since $f_n[1]=f_{n+1}$, and $\mathcal G(\beta)(\Delta\cB)\simeq\Delta\mathcal X$ in all our examples.
\end{proof}

As a particular case, we have the following analog of~\cite[Corollary 3.5]{ABB}.

\begin{corollary}\label{ncgraph}
	Consider a homotopy commuting triangle \raisebox{.65cm}{\xymatrix{\cA\ar^-{\nu}_-\sim[rr]&&\cB\\&\Upsilon\ar^-\phi[lu]\ar_-\psi[ru]&}} where the upward maps are assumed to belong to $\mathbf C_1^n(\Upsilon,\varnothing)$. Denote by $c:k[n]\to\mathcal F\Delta(\Upsilon)$ the non-degenerate structure on $\Upsilon$. We thus have homotopies $\phi(c)\underset{\eta_\cA}{\sim}0$ and $\psi(c)\underset{\eta_\cB}{\sim}0$. Assume that $\nu(\eta_\cA) \sim \eta_{\cB}$.
	Then $\cA\amalg_\Upsilon\cB\to\cB\in \mathbf C_1^{n+1}(\cA\amalg_\Upsilon\cB,\varnothing)$.
\end{corollary}

We are ready to prove the following, which specifies to~\cite[Lemma 5.33]{BCS} in the Calabi--Yau case.

\begin{proposition}
	Consider $\mathcal Z\in\scat{C}^n_1(\mathcal X,\mathcal Y)$, $\mathcal U\in\scat{C}^n_1(\varnothing,\mathcal X)$ and $\mathcal V\in\scat{C}^n_1(\mathcal X,\varnothing)$. Then the cospan\[
	\mathcal U\coprod_{\mathcal X}\mathcal V
	\rightarrow
	\mathcal U\coprod_{\mathcal X}\mathcal Z\coprod_{\mathcal X}\mathcal V
	\leftarrow
	\mathcal U\coprod_{\mathcal X}\mathcal Z\coprod_{\mathcal Y}\mathcal Z\coprod_{\mathcal X}\mathcal V
	\]
	belongs to $\scat{C}^{n+1}_1$. 
\end{proposition}

\begin{proof}
	We identify $\mathcal Z$ with $\mathcal Z^\mathrm{op}\in\scat{C}^n_1(\mathcal Y,\mathcal X)$ and consider $\mathcal Z\circ\mathcal Z\in\scat{C}^n_1(\mathcal X,\mathcal X)$. We can see $\mathcal Z$ as a $2$-cell $1_\mathcal X\Rightarrow\mathcal Z\circ\mathcal Z$, where $1_\mathcal X\in \scat{C}_1^n(\mathcal X,\mathcal X)$ is the codiagonal, via\[
	\xymatrix{
		&\mathcal X\ar[d]&\\
		\mathcal X\ar[ru]\ar[rd]&\mathcal Z&\mathcal X,\ar[lu]\ar[ld]\\
		&\mathcal Z\coprod_{\mathcal Y}\mathcal Z\ar[u]&
	}
	\] 
	as \[\left(
	\mathcal Z\leftarrow\mathcal X\coprod_{\mathcal X\coprod \mathcal X}\mathcal Z\coprod_{\mathcal Y}\mathcal Z\simeq \mathcal Z\coprod_{\mathcal X\coprod\mathcal Y}\mathcal Z\right)\in\mathbf C^{n+1}_1\]
	thanks to Corollary~\ref{ncgraph} (with $\phi=\psi=(\mathcal X\amalg\mathcal Y\to\mathcal Z)$ and $\nu=\mathrm{id}_\mathcal Z$).
	Pre and post horizontal compositions by $\mathcal V$ and $\mathcal U$ yield a $2$-cell \[
	\begin{tikzcd}
		\displaystyle  \varnothing \arrow[rrrrrrr, bend left=40, "{\displaystyle\mathcal U\circ1_\mathcal X\circ\mathcal V}", ""{name=U,inner sep=3pt,below}]
		\arrow[rrrrrrr, bend right=40, "\displaystyle\mathcal U\circ\mathcal Z\circ\mathcal Z\circ \mathcal V"{below}, ""{name=D,inner sep=3pt}]
		& &&&&&&\varnothing
		\arrow[Rightarrow, from=U, to=D, "~\displaystyle\mathcal U\circ \mathcal Z \circ \mathcal V"]
	\end{tikzcd}
	\]
	which gives the result.
\end{proof}


\end{document}